\documentclass[twoside, preprint]{imsart}
\usepackage[english]{babel}
\usepackage{parskip}
\usepackage[a4paper]{geometry}
\usepackage{amsmath,amsfonts,amsthm, amssymb}
\usepackage{graphicx}
\usepackage{subfigure}

\RequirePackage[colorlinks,citecolor=blue,urlcolor=blue]{hyperref}

\startlocaldefs
\newcommand{\E}{\mathbb{E}}
\newcommand{\Var}{\operatorname{Var}}
\newcommand{\1}{\textbf{1}}

\theoremstyle{plain}
\newtheorem{thm}{Theorem}[section]
\newtheorem{lem}[thm]{Lemma}

\newtheorem{cond}{Condition}
\newtheorem{innercustomthm}{Condition}
\newenvironment{condprime}[1]
  {\renewcommand\theinnercustomthm{#1}\innercustomthm}
  {\endinnercustomthm}

\graphicspath{{./Figures/}}

\endlocaldefs

\begin{document}

\begin{frontmatter}

\title{Conditions for Posterior Contraction in the Sparse Normal Means Problem} 
\runtitle{Conditions for the sparse normal means problem}

\begin{aug} 
\author{\fnms{S.L.}  \snm{van der Pas}\thanksref{t1}\ead[label=e1]{svdpas@math.leidenuniv.nl}},
  \author{\fnms{J.-B.} \snm{Salomond}\thanksref{t2}\ead[label=e2]{salomond@ceremade.dauphine.fr}}
  \and
  \author{\fnms{J.}  \snm{Schmidt-Hieber}%
  \ead[label=e3]{schmidthieberaj@math.leidenuniv.nl}%
  }

  \thankstext{t1}{Research supported by Netherlands Organization for Scientific
Research NWO.}
  \thankstext{t2}{Research supported by NWO VICI project `Safe Statistics'.}

  \runauthor{van der Pas, Salomond and Schmidt-Hieber}
  
  \address{Leiden University, Mathematical Institute, Niels Bohrweg 1, 2333 CA Leiden, The Netherlands\\ \printead{e1,e3}}
  \address{Universit\'e Paris Dauphine, Place du Mar\'echal DeLattre de Tassigny, 75016 Paris, France\\ \printead{e2}}

\end{aug}

\begin{abstract}
The first Bayesian results for the sparse normal means problem were proven for spike-and-slab priors. However, these priors are less convenient from a computational point of view. In the meanwhile, a large number of continuous shrinkage priors has been proposed. Many of these shrinkage priors can be written as a scale mixture of normals, which makes them particularly easy to implement. We propose general conditions on the prior on the local variance in scale mixtures of normals, such that posterior contraction at the minimax rate is assured. The conditions require tails at least as heavy as Laplace, but not too heavy, and a large amount of mass around zero relative to the tails, more so as the sparsity increases. These conditions give some general guidelines for choosing a shrinkage prior for estimation under a nearly black sparsity assumption. We verify these conditions for the class of priors considered in \cite{Ghosh2015}, which includes the horseshoe and the normal-exponential gamma priors, and for the horseshoe+, the inverse-Gaussian prior, the normal-gamma prior, and the spike-and-slab Lasso, and thus extend the number of shrinkage priors which are known to lead to posterior contraction at the minimax estimation rate. 
\end{abstract}



\end{frontmatter}

\section{Introduction}
In  the sparse normal means problem,  we wish to estimate a sparse vector $\theta$ based on a vector $X^n \in \mathbb{R}^n$, $X^n = (X_1, \ldots, X_n)$, generated according to the model
\begin{equation*}
X_i = \theta_i + \varepsilon_i, \quad i = 1, \ldots, n,
\end{equation*}
where the $\varepsilon_i$ are independent standard normal variables. The vector of interest $\theta$ is sparse in the \emph{nearly black} sense, that is, most of the parameters are zero. We wish to separate the signals (nonzero means) from the noise (zero means). Applications of this model include image reconstruction and nonparametric function estimation using wavelets \cite{Johnstone2004}. 

The model is an important test case for the behaviour of sparsity methods, and has been well-studied. A great variety of frequentist and Bayesian estimators has been proposed, and the popular Lasso \cite{Tibshirani1996} is included in both categories. It is but one example of many approaches towards recovering $\theta$; restricting ourselves to Bayesian methods, other approaches include shrinkage priors such as the spike-and-slab type priors studied by \cite{Johnstone2004, Castillo2012} and \cite{Castillo2015}, the normal-gamma prior  \cite{Griffin2010},  non-local priors \cite{Johnson2010}, the Dirichlet-Laplace prior \cite{Bhattacharya2014}, the horseshoe \cite{Carvalho2010}, the horseshoe+ \cite{Bhadra2015} and the spike-and-slab Lasso \cite{Rockova2015}. 

Our goal is twofold: \emph{recovery} of the underlying mean vector, and \emph{uncertainty quantification}. The benchmark for the former is estimation at the minimax rate. In a Bayesian setting, the typical choice for the estimator is some measure of center of the posterior distribution, such as the posterior mean, mode or median. For the purpose of uncertainty quantification, the natural object to use is a credible set. In order to obtain credible sets that are narrow enough to be informative, yet not so narrow that they neglect to cover the truth, the posterior distribution needs to contract to its center at the same rate at which the estimator approaches the truth. 

For recovery, spike-and-slab type priors give optimal results (\cite{Johnstone2004, Castillo2012, Castillo2015}). These priors assign independently to each component a mixture of a point mass at zero and a continuous prior. Due to the point mass, spike-and-slab priors shrink small coefficients to zero. The advantage is that the full posterior has optimal model selection properties but this comes at the prize of, in general, too narrow credible sets. Another drawback of spike-and-slab methods is that they are computationally expensive although the complexity is much better than what has been previously believed (\cite{Yang2015}). 

Thus, we might ask whether there are priors which are smoother and shrink less than the spike-and-slab but still recover the signal with a (nearly) optimal rate. A naive choice would be to consider the Laplace prior $\propto e^{-\lambda \|\theta\|_1}$ with $\|\theta\|_1=\sum_{i=1}^n |\theta_i|,$ since in this case the maximum a posteriori (MAP)  estimator coincides with the Lasso, which is known to achieve the optimal rates for sparse signals. In \cite{Castillo2015}, Section 3, it was shown that although the MAP-estimator has good properties, the full posterior spreads a non-negligible amount of mass over large neighborhoods of the truth leading to recovery rates that are sub-optimal by a polynomial factor in $n.$ This example shows that if the prior does not shrink enough, we loose the recovery property of the posterior. 

Recently, shrinkage priors were found that are smoother than the spike-and-slab but still lead to (near) minimax recovery rates. Up to now, optimal recovery rates have been established for the horseshoe prior \cite{vanderPas2014}, horseshoe-type priors with slowly varying functions \cite{Ghosh2015}, the empirical Bayes procedure of  \cite{Martin2014}, the spike-and-slab Lasso \cite{Rockova2015}, and the Dirichlet-Laplace prior, although the latter result only holds under a restriction on the signal size \cite{Bhattacharya2014}. Finding smooth shrinkage priors with theoretical guarantees remains an active area of research.

The question arises which features of the prior lead to posterior convergence at the minimax estimation rate. Qualitative discussion on this point is provided by \cite{Carvalho2010}. Intuitively, a prior should place a large amount of mass near zero to account for the zero means, and have heavy tails to counteract the shrinkage effect for the nonzero means. In the present article, we make an attempt to quantify the relevant properties of a prior, by providing general conditions ensuring posterior concentration at the minimax rate, and showing that a large number of priors (including the ones listed above) meets these conditions. 

We study scale mixtures of normals, as many shrinkage priors proposed in the literature are contained in this class and provide general conditions on the prior on the local variance such that posterior concentration at the minimax estimation rate is guaranteed. These conditions are general enough to recover the already known results for the horseshoe prior, the horseshoe-type priors with slowly varying functions and the spike-and-slab Lasso, and to demonstrate that the horsehoe+ \cite{Bhadra2015}, inverse-Gaussian prior \cite{Caron2008} and the normal-gamma prior \cite{Caron2008, Griffin2010} lead to posterior concentration at the correct rate as well. Our conditions in essence mean that a sparsity prior should have tails that are at least as heavy as Laplace, but not too heavy, and there should be a sizable amount of mass close to zero relative to the tails, especially when the underlying vector is very sparse. 

This paper is organized as follows. We state our main result, providing conditions on sparsity priors such that the posterior contracts at the minimax rate in Section \ref{sec:main}. We then show, in Section \ref{sec:examples}, that these conditions hold for the class of priors of \cite{Ghosh2015}, as well as for the horseshoe+, the inverse-Gaussian prior, the normal-gamma prior, and the spike-and-slab Lasso. A simulation study is performed in Section \ref{sec:simulation}, and we conclude with a Discussion. All proofs are given in  Appendix \ref{sec:proofs}.

\emph{Notation}.
Denote the class of nearly black vectors by $\ell_0[p_n] = \{\theta \in \mathbb{R}^n \ : \sum_{i=1}^n \textbf{1}\{\theta_i \neq 0\} \leq p_n\}$. The minimum $\min\{a,b\}$ is given by $a \wedge b$. The standard normal density is denoted by $\phi$, its cdf by $\Phi$, and we set $\Phi^c(x) = 1 - \Phi(x)$. The norm $\|\cdot\|$ is the $\ell_2$-norm.

\section{\label{sec:main}Main results}
Each coefficient $\theta_i$ receives a scale mixture of normals as a prior:
\begin{equation}\label{eq:theprior}
\theta_i \mid \sigma^2_i \sim \mathcal{N}(0, \sigma_i^2), \quad \quad \sigma_i^2 \sim \pi(\sigma_i^2), \quad i = 1, \ldots, n, 
\end{equation}
where $\pi:[0,\infty) \to [0,\infty)$ is a density on the positive reals. While $\pi$ might depend on further hyperparameters, no additional priors are placed on such parameters, rendering the coefficients independent \emph{a posteriori}. The goal is to obtain conditions on $\pi$  such that posterior concentration at the minimax estimation rate is guaranteed.

We use the coordinatewise posterior mean to recover the underlying mean vector. By Tweedie's formula \cite{Robbins1956}, the posterior mean for $\theta_i$ given an observation $x_i$ is equal to $x_i + \frac{d}{dx} \log p(x_i) $, where $p(x_i)$ is the marginal distribution of $x_i$. The posterior mean for parameter $\theta_i$ is thus given by $X_im_{X_i}$, where $m_{x}: \mathbb{R} \to [0,1]$ is
\begin{equation}\label{eq:defmx}
m_x := \frac{\int_0^1 z (1-z)^{-3/2} e^{ \frac{x^2}{2}z} \pi\big(\frac{z}{1-z}\big) dz}{\int_0^1  (1-z)^{-3/2}e^{\frac{x^2}{2}z} \pi\big(\frac{z}{1-z}\big) dz }
= \frac{\int_0^\infty u (1+u)^{-3/2} e^{ \frac{x^2u}{2+2u}} \pi(u) du}{\int_0^\infty (1+u)^{-1/2} e^{ \frac{x^2u}{2+2u}} \pi(u) du}.
\end{equation}
We denote the estimate of the full vector $\theta$ by  $\widehat\theta = (X_1m_{X_1}, \ldots, X_nm_{X_n})$. An advantage of scale mixtures of normals as shrinkage priors over spike-and-slab-type priors, is that the posterior mean can be represented as the observation multiplied by \eqref{eq:defmx}. The ratio \eqref{eq:defmx} can be computed via integral approximation methods such as a quadrature routine. See \cite{Polson2012}, \cite{Polson2012-2} and \cite{vanderPas2014} for more discussion on this point in the context of the horseshoe. 

Our main theorem, Theorem \ref{thm:main}, provides three conditions on $\pi$ under which a prior of the form \eqref{eq:theprior} leads to an upper bound on the posterior contraction rate of the order of the minimax rate. We first state and discuss the conditions. In addition, we present stronger conditions that are easier to verify. Condition \ref{cond1} is required for our bounds on the posterior mean and variance for the nonzero means. The remaining two are used for the bounds for the zero means. 

The first condition involves a class of regularly varying functions. Recall that a function $\ell$ is called \emph{regular varying (at infinity)} if  for any $a>0,$ the ratio $\ell(au)/\ell(u)$ converges to the same non-zero limit as $u\rightarrow \infty.$ For our estimates, we need a slightly different notion, that will be introduced next. We say that a function $L$ is {\it uniformly regular varying,} if there exist constants $R, u_0\geq 1,$ such that
\begin{align}
	\frac{1}{R} \leq \frac{L(au)}{L(u)} \leq R, \quad \text{for all} \ a\in [1,2], \ \text{and all} \ u\geq u_0.
	\label{eq.unif_reg_vary_at_zero_prop}
\end{align} 
In particular, $L(u) = u^b,$ and $L(u)=\log^b(u)$ with $b\in \mathbb{R}$ are uniformly regular varying (take for example $R=2^{|b|}$ and $u_0=2$). An example of a function that is not uniformly regular varying is $L(u)=e^{u}.$ From the definition, we can easily deduce the following properties of functions that are uniformly regular varying. Firstly, $u\mapsto L(u)$ is on $[u_0, \infty)$ either everywhere positive or everywhere negative.  If $L$ is uniformly regular varying then also $u\mapsto 1/L(u)$ and if $L_1$ and $L_2$ are uniformly regular varying, then also their product $L_1L_2.$

We are now ready to present Condition \ref{cond1}, and the stronger Condition \ref{cond1prime}, which implies Condition 1, as shown in Lemma \ref{lem:cond1prime}.\\

\begin{cond}\label{cond1}
For some $b\geq 0,$ we can write $z\mapsto \pi(u) = L_n(u) e^{-bu},$ where $L_n$ is a function that satisfies \eqref{eq.unif_reg_vary_at_zero_prop} for some $R, u_0\geq 1$ which do not depend on $n.$ Suppose further that there are constants $C', K , b'\geq 0$ and $u_* \geq 1,$ such that 
\begin{align}
	  C' \pi(u) \geq \Big(\frac{p_n}{n}\Big)^{K} e^{-b'u}  \quad \text{for all} \ u\geq u_*.
	\label{eq.assump_on_lb_Ln}
\end{align}
\end{cond}

\begin{condprime}{1'}\label{cond1prime}
Consider a global-local scale mixture of normals: 
\begin{equation} \label{eq:globallocal}
\theta_i \mid \sigma_i^2, \tau^2 \sim \mathcal{N}(0, \sigma_i^2\tau^2), \quad \sigma_i^2 \sim \widetilde \pi(\sigma_i^2), \quad i = 1, \ldots, n.
\end{equation}
Assume that $\widetilde \pi$ is a uniformly regular varying function which does not depend on $n$, and $\tau = (p_n/n)^\alpha$ for $\alpha \geq 0$. \\
\end{condprime}

Condition \ref{cond1} assures that the posterior recovers nonzero means with the optimal rate. Thus, the condition can be seen as a sufficient condition on the tail behavior of the density $\pi$ for $\ell^2$-recovery. The tail may decay exponentially fast, which is consistent with the conditions found on the `slab' in the spike-and-slab priors discussed by \cite{Castillo2012}. In general, $\pi$ will depend on $n$ through a hyperparameter. Condition \ref{cond1} requires that the $n$ dependence behaves roughly as a power of $p_n/n.$

In the important special case where each $\theta_i$ is drawn independently from a global-local scale mixture, Condition \ref{cond1} is satisfied whenever the density on the local variance is uniformly regular varying, as stated in Condition \ref{cond1prime}. Below, we give the conditions on $\pi$ that guarantee posterior shrinkage at the minimax rate for the zero coefficients. The first condition ensures that the prior $\pi$ puts some finite mass on values between $[0,1].$\\

\begin{cond}\label{cond2}
Suppose that there is a constant $c>0$ such that $\int_0^1 \pi(u) du \geq c.$
\end{cond}

We turn to Condition \ref{cond3} which describes the decay of $\pi$ away from a neighborhood of zero. To state the condition it will be convenient to write
\begin{align}
	s_n:=\frac{p_n} n \log(n/p_n).
	\label{eq.sn_def}
\end{align}

\begin{cond} \label{cond3}
Let $b_n=\sqrt{\log(n/p_n)}$ and assume that there is a constant $C,$ such that
\begin{align*}
	\int_{s_n}^\infty \Big( u \wedge \frac{b_n^3}{\sqrt{u}} \Big) \pi(u) du
	+ b_n \int_1^{b_n^2} \frac{\pi(u)}{\sqrt{u}} du
	\leq C s_n.
\end{align*}
\end{cond}

In order to allow for many possible choices of $\pi,$ the tail condition involves several terms. It is surprising that some control on the interval $\int_{s_n}^1 u\pi(u) du$ is needed. But this turns out to be sharp. Theorem \ref{thm.necessary} proves that if we would relax the condition to $\int_{s_n}^1 u\pi(u) du \lesssim t_n$ for an arbitrary rate $t_n \gg s_n,$ then there is a prior that satisfies all the other conditions needed for  the zero coefficients, but which does not concentrate at the minimax rate.

Below we state two stronger conditions, each of which  obviously imply Condition \ref{cond2} and Condition \ref{cond3} for sparse signals, that is, $p_n =o(n).$\\

\begin{condprime}{A}\label{condA}
Assume that there is a constant $C,$ such that 
\begin{align*}
	\pi(u) \leq \frac{C}{u^{3/2}} \frac{p_n}{n} \sqrt{\log(n/p_n)}, \quad \text{for all} \ u\geq s_n.
\end{align*}
\end{condprime}

\begin{condprime}{B}\label{condB}
Assume that there is a constant $C,$ such that 
\begin{align*}
	\int_{s_n}^\infty \pi(u) du \leq \frac{Cp_n}{n}.
\end{align*}
\end{condprime}

In this case, even a stronger version of Condition \ref{cond2} holds in the sense that nearly all mass is concentrated in the shrinking interval $[0,s_n].$ Notice that Condition \ref{cond3} does not imply Condition \ref{cond2} in general. If for example $\pi$ is a point mass at $n^2,$ then, Condition \ref{cond3} holds but Condition \ref{cond2} does not. Condition \ref{cond1} and Condition \ref{cond3} depend on the relative sparsity $p_n/n.$ Indeed, Condition \ref{cond1} becomes weaker if the signal is more sparse and at the same time Condition \ref{cond3} becomes stronger. This matches intuition, as the prior should shrink more in this case and thus the assumptions that are responsible for the shrinkage effect should become stronger.

\begin{figure}[h]
\begin{center}
\subfigure
{
\label{fig:p-hs}
\includegraphics[width = 0.3\textwidth]{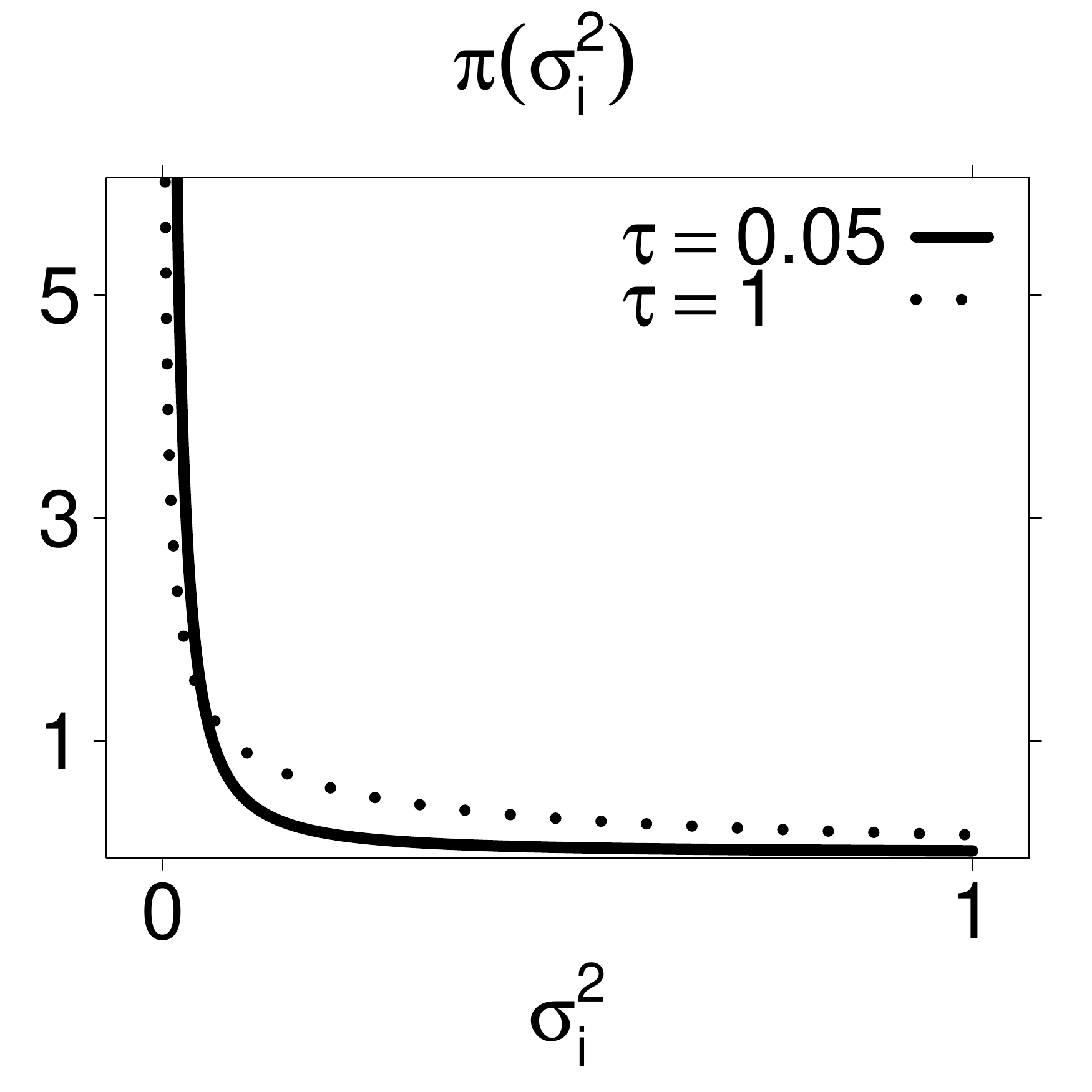}
}
\subfigure
{
\label{fig:p-ig}
\includegraphics[width = 0.3\textwidth]{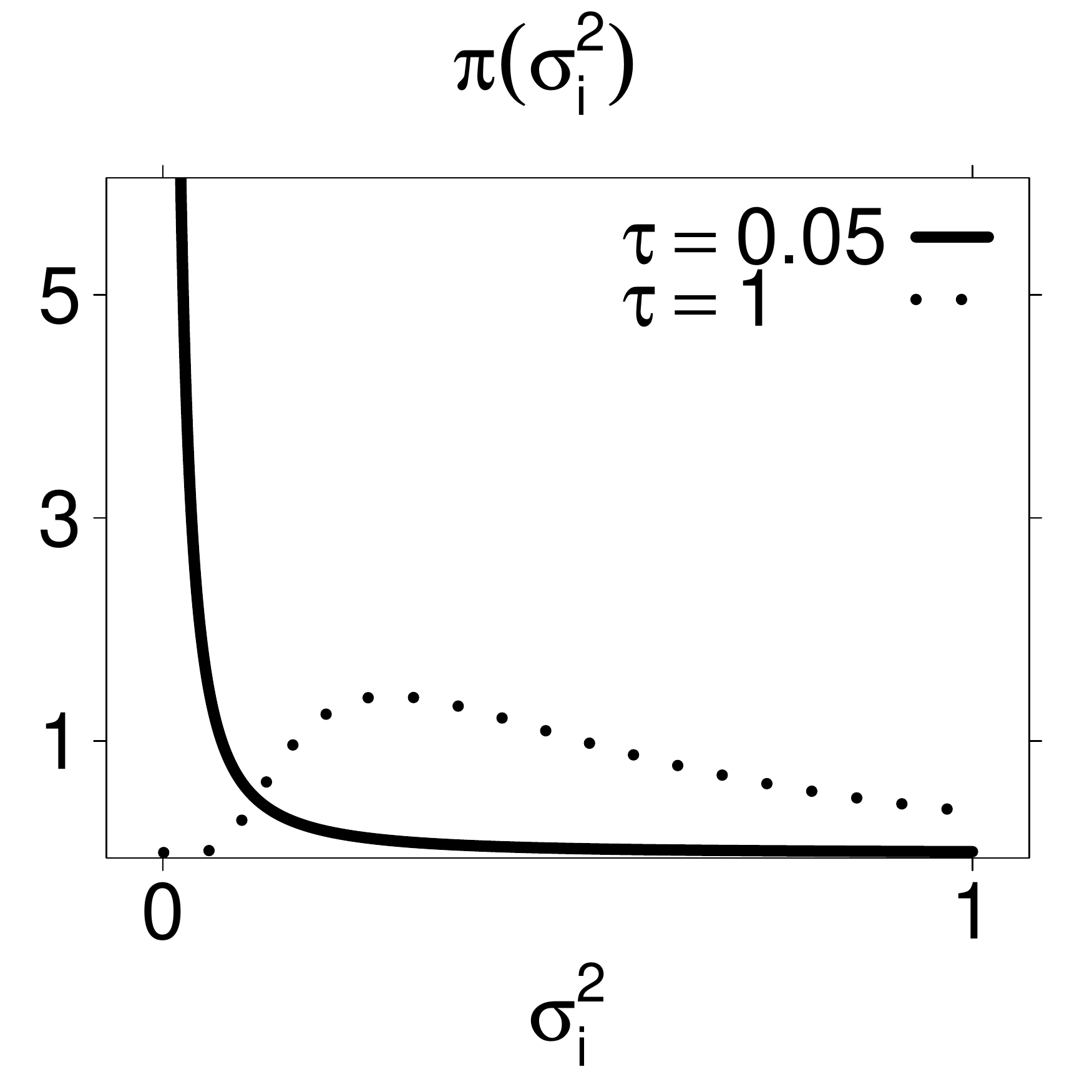}
}
{
\label{fig:p-ng}
\includegraphics[width = 0.3\textwidth]{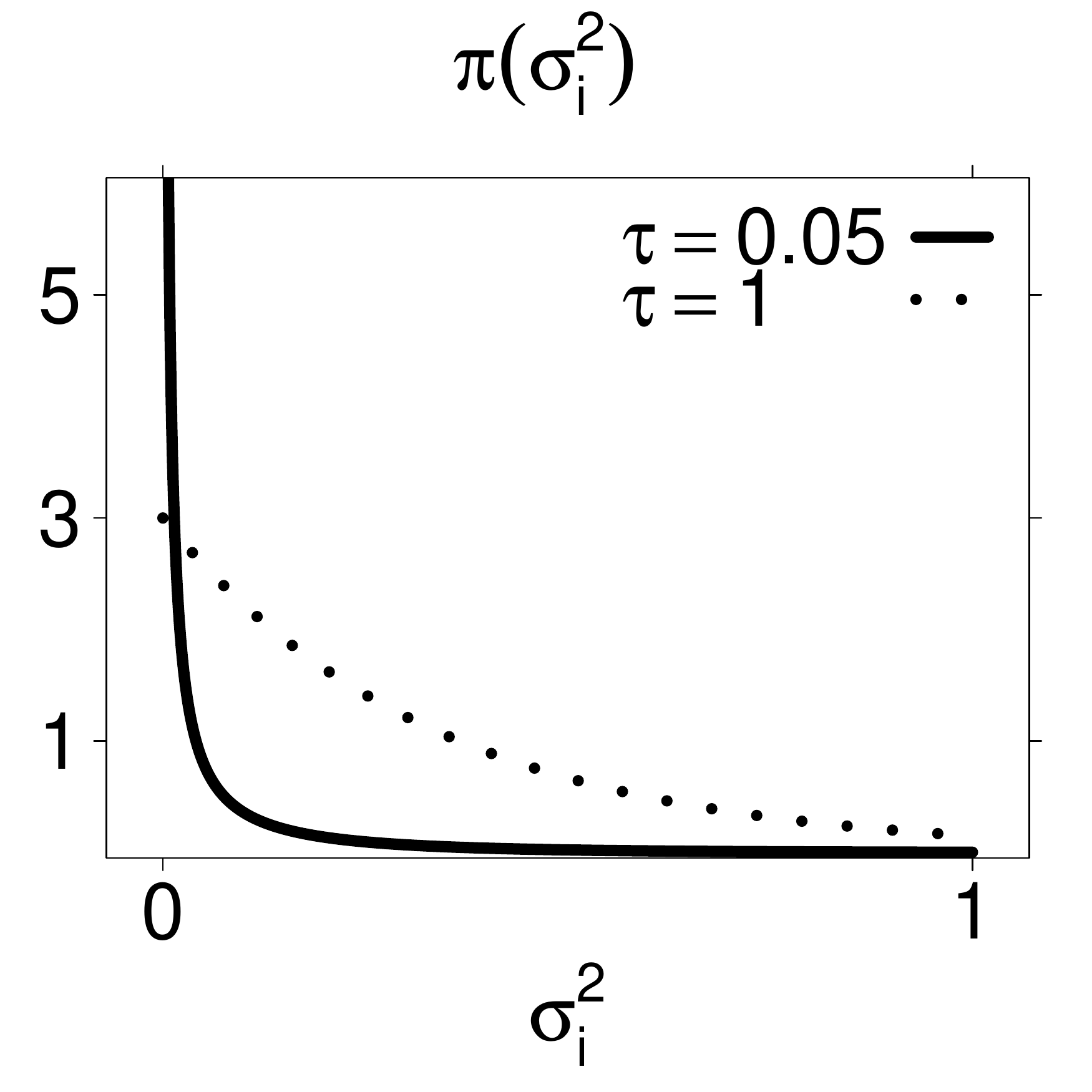}
}
{
\label{fig:theta-hs}
\includegraphics[width = 0.3\textwidth]{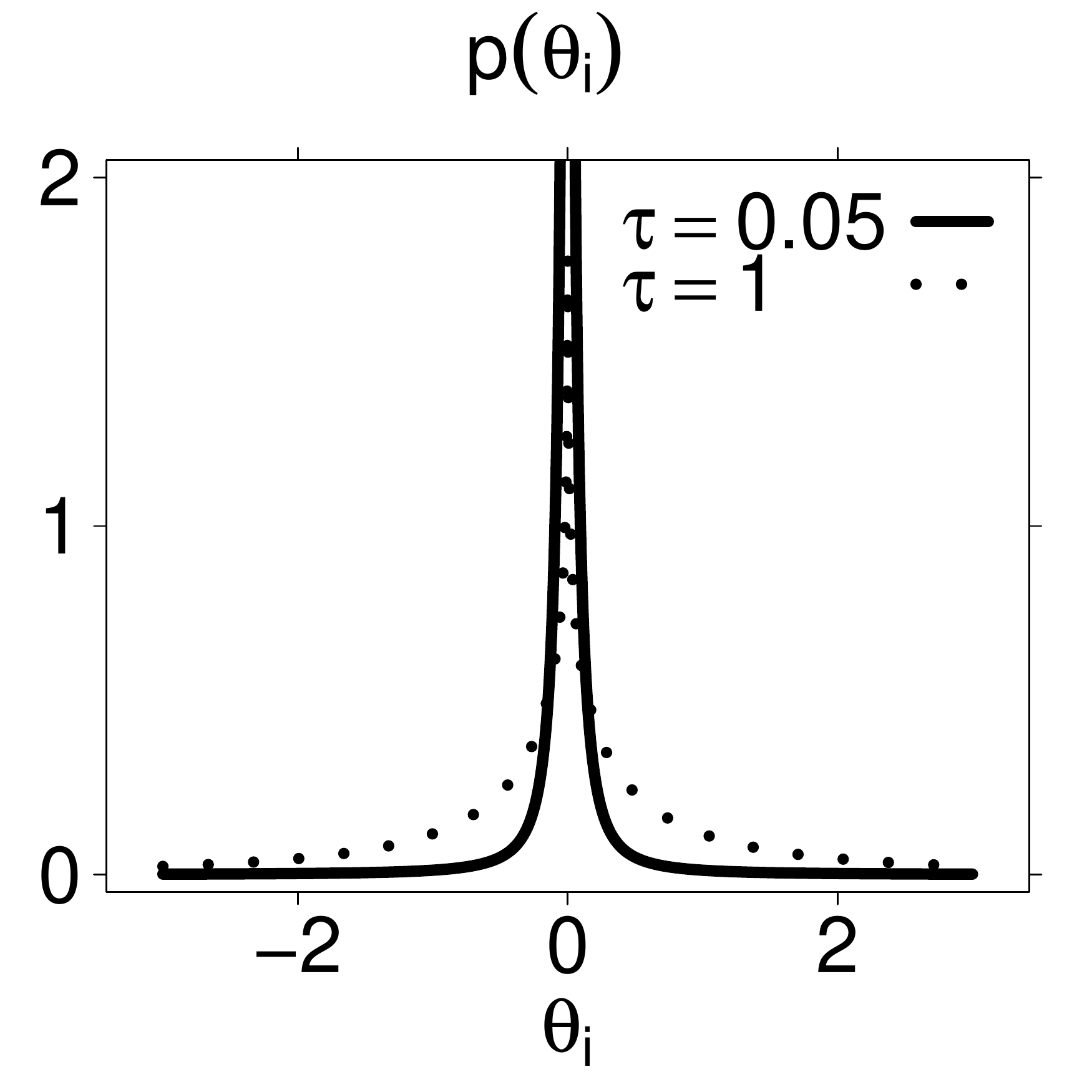}
}
\subfigure
{
\label{fig:theta-ig}
\includegraphics[width = 0.3\textwidth]{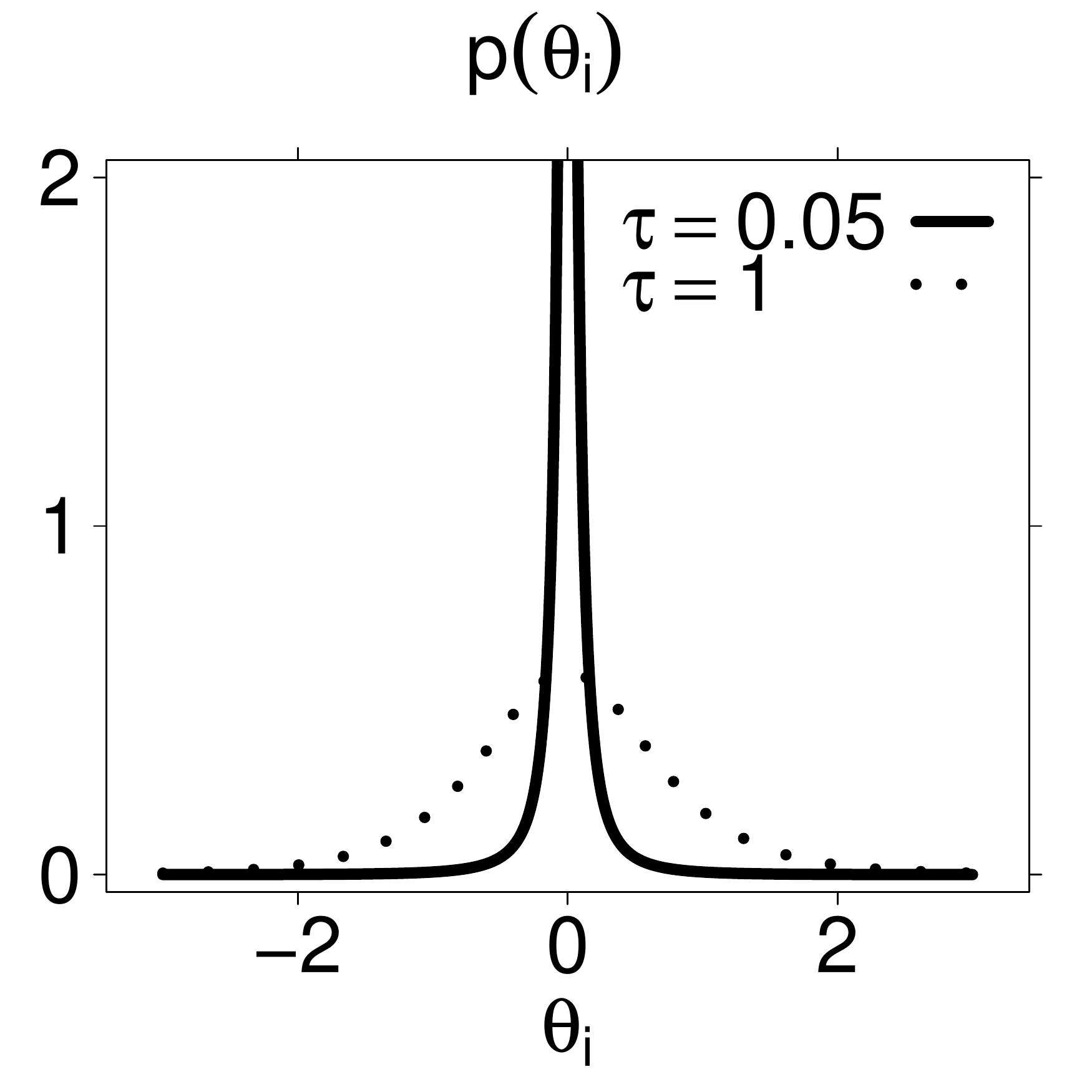}
}
{
\label{fig:theta-ng}
\includegraphics[width = 0.3\textwidth]{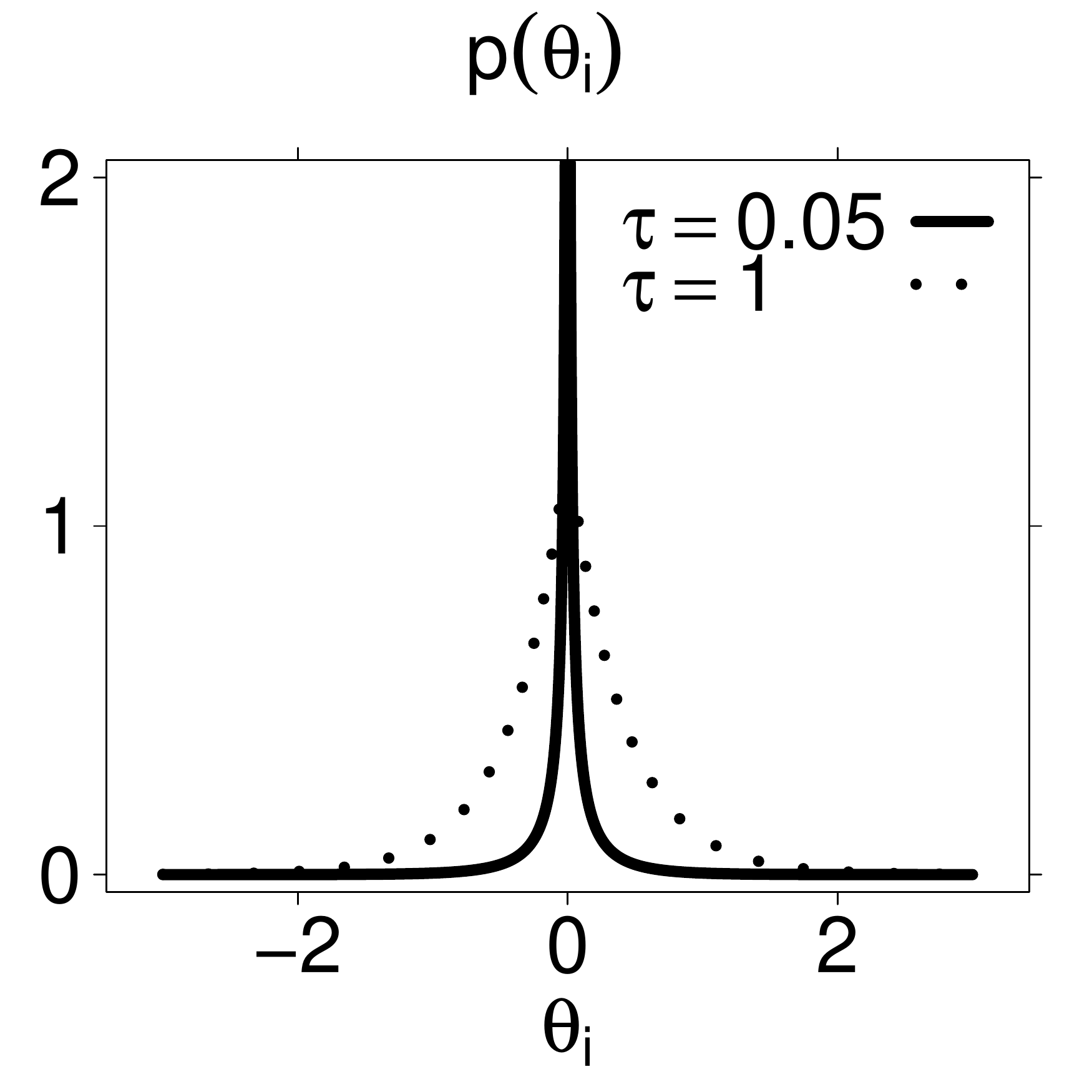}
}
\caption{Plots of priors on the local variance (first row) and the corresponding parameters (second row). From left to right: horseshoe, Inverse-Gaussian with $a = 1/2, b = 1$, and normal gamma with $\beta = 3$. The parameter $\tau$, which in practice should be of the order $p_n/n$, is taken equal to 1 (dashed line) and 0.05 (solid line).  } 
\label{fig:effecttau}
\end{center} 
\end{figure}

Figure \ref{fig:effecttau} presents plots of the priors $\pi$ on the local variance, and the corresponding priors on the parameters $\theta_i$, for three priors for which the three conditions are verified in Section \ref{sec:examples}: the horseshoe, inverse-Gaussian, and normal-gamma. The parameter $\tau$, in the notation of Section \ref{sec:examples}, should be thought of as the sparsity level $p_n/n$. Figure \ref{fig:effecttau} shows that the priors start to resemble each other when $\tau$ is decreased. If the setting is more sparse, corresponding to more zero means, the mass of the prior $\pi$ on $\sigma_i^2$ concentrates around zero, leading to a higher peak at zero in the prior density on $\theta_i$.

We now present our main result. The minimax estimation rate for this problem, under $\ell_2$ risk, is given by $2p_n\log(n/p_n)$  \cite{Donoho1992}. We write $\theta^0=(\theta_i^0)_{i=1,\ldots,n}$ and consider posterior concentration of the zero and non-zero coefficients separately. Asymptotics always refers to $n\rightarrow \infty.$ \\

\begin{thm}\label{thm:main}
Work under model $X^n \sim \mathcal{N}(\theta_0, I_n)$ and assume that the prior is of the form \eqref{eq:theprior}. Suppose further that $p_n = o(n)$ and let $M_n$ be an arbitrary positive sequence tending to $+\infty$. Under Condition \ref{cond1},
\begin{equation*}
\sup_{\theta_0 \in \ell_0[p_n]} \E_{\theta_0}\Pi\big( \theta \ : \sum_{i:\theta_i^0 \neq 0} (\theta_i - \theta_i^0)^2 > M_n p_n\log(n/p_n) \ \bigr| \ X^n \big)  \to 0
\end{equation*}
and
\begin{equation*}
\sup_{\theta_0 \in \ell_0[p_n]}  \E_{\theta_0} \sum_{i:\theta_i^0 \neq 0} (\widehat \theta_i - \theta_i^0)^2 \lesssim p_n\log(n/p_n).
\end{equation*}
Under Condition \ref{cond2} and Condition \ref{cond3} (or either Condition \ref{condA} or \ref{condB}),
\begin{equation*}
\sup_{\theta_0 \in \ell_0[p_n]} \E_{\theta_0}\Pi\big( \theta \ : \sum_{i:\theta_i^0 = 0} \theta_i^2 > M_n p_n\log(n/p_n) \ \bigr| \ X^n \big)  \to 0
\end{equation*}
and
\begin{equation*}
\sup_{\theta_0 \in \ell_0[p_n]}  \E_{\theta_0} \sum_{i:\theta_i^0 = 0} \widehat \theta_i ^2 \lesssim p_n\log(n/p_n).
\end{equation*}
Thus, under Conditions \ref{cond1}-\ref{cond3} (or Condition \ref{cond1} with either Condition \ref{condA} or \ref{condB}),
\begin{equation*}
\sup_{\theta_0 \in \ell_0[p_n]} \E_{\theta_0}\Pi\big( \theta \ : \|\theta - \theta_0\|^2 > M_n p_n\log(n/p_n) \ \bigr| \ X^n \big) \to 0
\end{equation*}
and
\begin{equation*}
\sup_{\theta_0 \in \ell_0[p_n]}  \E_{\theta_0} \big\|\widehat \theta-\theta_0\big\|_2^2 \lesssim p_n\log(n/p_n).
\end{equation*}
\end{thm}

The statement is split into zero and non-zero coefficients of $\theta_0$ in order to make the dependence on the conditions explicit. Indeed, posterior concentration of the non-zero coefficients follows from Condition \ref{cond1} and posterior concentration for the zero-coefficients is a consequence of Conditions \ref{cond2} and \ref{cond3}. It is well-known that posterior concentration at rate $\epsilon_n$ implies existence of a frequentist estimator with the same rate (cf. \cite{Ghosal2000}, Theorem 2.5). Thus, the rate of contraction around the true mean vector $\theta_0$ must be sharp. This also means that credible sets computed from the posterior cannot be too so large as to be uninformative, an effect that, as discussed in the introduction, occurs for the Laplace prior connected to the Lasso. If one wishes to use a credible set centered around the posterior mean, then its radius might still be too small to cover the truth. The first step towards guarantees on coverage is a lower bound on the posterior variance. Such a lower bound was obtained for the horseshoe in \cite{vanderPas2014}, and for priors very closely resembling the horseshoe in \cite{Ghosh2015}. No such results have been obtained so far for priors on $\sigma_i^2$ that have a tail of a different order than $(\sigma_i^2)^{-3/2}$. This is a delicate technical issue that we will not pursue further here.

The results also indicates how to build adaptive procedures. The method does not require explicit knowledge of $p_n$ but in order to get minimax concentration rates, we need to find priors that satisfy the conditions of Theorem \ref{thm:main}. Consider for example the prior defined as
\begin{align*}
	\pi(u) :=\frac{1}{u^{3/2}} \frac{\sqrt{\log n}}{n} , \quad \text{for all} \ u\geq \frac{\sqrt{\log n}}{n}
\end{align*}
and the remaining mass is distributed arbitrarily on the interval $[0, \sqrt{\log n}/n).$ Thus Condition \ref{condA} holds for any $1\leq p_n=o(n)$ and thus also Condition \ref{cond2} and Condition \ref{cond3}. Whenever we impose an upper bound $p_n\leq n^{1-\delta}$ with $\delta>0,$ then also Condition \ref{cond1} holds and thus Theorem \ref{thm:main} follows. This shows that in principle priors can be constructed that adapt over the whole range of possible sparsity levels and lead to some theoretical guarantee. From a practical point, however, these methods shrink to much and have to little mass in the tails. A better procedure would be to get a rough estimate of the relative sparsity $p_n/n$ in a first step and then to use a prior that lies on the "boundary" of the conditions in the sense that the both sides in the inequality of Condition \ref{cond3} are of the same order. An empirical Bayes procedure that first estimates the sparsity was found to work well in \cite{vanderPas2014}, arguing along the lines of \cite{Johnstone2004}. The sparsity level estimator counts the number of observations that are larger than the `universal threshold' of $\sqrt{2\log n}.$ Similar results are likely to hold in our setting, as long as the posterior mean is monotone in the parameter that is taken to depend on $p_n$.

\subsection{Necessary conditions}
\label{sec.necessary}

The imposed conditions are nearly sharp. To see this, consider the Laplace prior, where each $\theta_i$ is drawn independently from a Laplace distribution with parameter $\lambda.$ It is well-known that the Laplace distribution with parameter $\lambda$ can be represented as a scale mixture of normals where the mixing density is exponential with parameter $\lambda^2$ (cf. \cite{andrews1974} or \cite{park2008}, Equation (4)). Thus, the Laplace prior fits our framework \eqref{eq:theprior} with $\pi(u)=\lambda^2e^{-\lambda^2 u},$ for $u\geq 0.$ As mentioned in the introduction, the MAP-estimator of this prior is the Lasso but the full posterior does not shrink at the minimax rate. Indeed, Theorem 7 in \cite{Castillo2015} shows that if the true vector is zero, then, the posterior concentration rate has the lower bound $n/\lambda^2$ for the squared $\ell^2$-norm provided that $1\leq \lambda =o(\sqrt{n}).$ This should be compared to the optimal minimax rate $\log n$ (the rate for sparsity zero is the same as the rate for sparsity $p_n=1$). Thus, the lower bound shows that the rate is sub-optimal as long as 
\begin{align}
	\lambda \ll \sqrt{\frac{n}{\log n}}.
	\label{eq.LASSO_cond}
\end{align}
If $\lambda \gtrsim \sqrt{n/\log n},$ the lower bound is not sub-optimal anymore, but in this case, the non-zero components cannot be recovered with the optimal rate. The lower bound shows that the posterior does not shrink enough if $\lambda$ is not taken to be huge and thus either Condition \ref{cond2} or Condition \ref{cond3} must be violated, as these are the two conditions that guarantee shrinkage of the zero mean coefficients. 

Obviously, $\int_0^1 \pi(u) du \geq \int_0^1 e^{-u}du>0$ for $1\leq \lambda$ and thus Condition \ref{cond2} holds. For Condition \ref{cond3} notice that the integral can be split into the integral $\int_0^1 u\pi(u)du $ plus an integral over $[1,\infty)$ Now, if $\lambda$ tends to infinity faster than a polynomial order in $n$ then the integral over $[1,\infty)$ is exponentially small in $n.$ Thus Condition \ref{cond3} must fail because the integral over $\int_{s_n}^1u\pi(u) du$ is of a larger order than $s_n=n^{-1}\log n.$ To see this, observe that for $\lambda \leq \sqrt{n/\log n},$
\begin{align*}
	\int_{s_n}^1u \lambda^2 e^{-\lambda^2 u} du = \frac{1}{\lambda^2} \int_{s_n\lambda^2}^{\lambda^2} ve^{-v} dv
	\geq \frac{1}{\lambda^2} \int_1^{\lambda^2}e^{-v} dv \gtrsim \frac{1}{\lambda^2}.
\end{align*}
Now, we see that Condition \ref{cond3} fails if and only if \eqref{eq.LASSO_cond} holds. Indeed, if $\lambda \ll \sqrt{n/\log n},$ then the r.h.s. is of larger order than $s_n$ and if $\lambda \asymp \sqrt{n/\log n},$ then, Condition \ref{cond3} holds. This shows that this bound is sharp. 

In order to state this as a formal result, let us introduce the following modification of Condition \ref{cond3}. Let $\kappa_n$ denote an arbitrary positive sequence.\\

\begin{condprime}{3($\kappa_n$)}\label{cond3modified}
Let $b_n=\sqrt{\log(n/p_n)}$ and assume that there is a constant $C,$ such that
\begin{align*}
	\kappa_n \int_{s_n}^1 u\pi(u) du +\int_{1}^\infty \Big( u \wedge \frac{b_n^3}{\sqrt{u}} \Big) \pi(u) du
	+ b_n \int_1^{b_n^2} \frac{\pi(u)}{\sqrt{u}} du
	\leq C s_n.
\end{align*}
\end{condprime}
In particular, we recover Condition \ref{cond3} for $\kappa_n=1.$\\

\begin{thm}
\label{thm.necessary}
Work under model $X^n \sim \mathcal{N}(\theta_0, I_n)$ and assume that the prior is of the form \eqref{eq:theprior}. For any positive sequence $(\kappa_n)_n$ tending to zero, there exists a prior $\pi$ satisfying Condition \ref{cond2} and Condition \ref{cond3modified} for $p_n=1$ and a positive sequence $(M_n)_n$ tending to infinity, such that
\begin{equation}\label{eq.thm_nec_concl}
	\E_{\theta_0=0} \Pi\big( \theta \ : \|\theta\|_2^2 \leq M_n \log(n) \ \bigr| \ X^n \big) \to 0, \quad \text{as} \ \ n\to \infty.
\end{equation}
\end{thm}

This theorem shows that the posterior puts asymptotically all mass outside an $\ell^2$-ball with radius $ M_n \log(n)\gg \log (n)$ and is thus suboptimal. The proof can be found in the appendix.

\section{\label{sec:examples}Examples}
In this section, Conditions \ref{cond1}-\ref{cond3} are verified for the  horseshoe-type priors considered by \cite{Ghosh2015} (which includes the horseshoe and the normal-exponential gamma), the horseshoe+, the inverse-Gaussian prior,  the normal-gamma prior, and the spike-and-slab Lasso. There are, to the best of our knowledge, no existing results yet showing that the horseshoe+, the inverse-Gaussian and the normal-gamma priors lead to posterior contraction at the minimax estimation rate. Posterior concentration for the horseshoe and horseshoe-type priors were already established in \cite{vanderPas2014} and \cite{Ghosh2015}, and for the spike-and-slab Lasso in \cite{Rockova2015} . Here, we obtain the same results but thanks to Theorem \ref{thm:main} the proofs become extremely short. In addition, we can show that a restriction on the class of priors considered by \cite{Ghosh2015} can be removed.

\subsection{\label{sec:exampleGhosh} Global-local scale mixtures of normals}
\label{subsec:ex:ghosh}
In \cite{Ghosh2015}, the priors under consideration are normal priors with random variances of the form
\begin{equation*}
\theta_i \mid \sigma_i^2, \tau^2 \sim \mathcal{N}(0, \sigma_i^2\tau^2), \quad \sigma_i^2 \sim \pi'(\sigma_i^2), \quad i = 1, \ldots, n,
\end{equation*}
for priors $\pi'$ with density given by
\begin{equation} \label{eq:Ghoshform}
\pi'(\sigma_i^2) = K\frac{1}{ (\sigma_i^2)^{a+1}}L(\sigma_i^2),
\end{equation}
where $K>0$ is a constant and $L : (0, \infty) \to (0, \infty)$ is a non-constant, \emph{slowly varying} function, meaning that there exist $c_0, M \in (0, \infty)$ such that $L(t) > c_0$ for all $t \geq t_0$ and $\sup_{t \in (0, \infty)} L(t) \leq M$. \cite{Ghosh2015} prove an equivalent of Theorem \ref{thm:main} for these priors, for $a \in [1/2, 1)$ and  $\tau =(p_n/n)^\alpha$ with $\alpha\geq 1$.

The horseshoe prior, with $\pi(u) = (\pi\tau)^{-1}u^{-1/2}(1 + u/\tau^2)^{-1}$,  is contained in this class of priors, by taking $a = 1/2$, $L(t) = t/(1+t)$, and $K = 1/\pi$. This class also contains the normal-exponential-gamma priors of \cite{Griffin2005}, for which $\pi(u) = \lambda/\gamma^2(1 + u/\gamma^2)^{-(\lambda+1)}$ with parameters $\lambda, \gamma>0.$ This class of priors is of the form \eqref{eq:Ghoshform} for the choice $\tau = \gamma$, $a = \lambda$ and $L(t) = (t/(1+t))^{1+\lambda}$. In \cite{Ghosh2015}, it is stated that the three parameter beta normal mixtures, the generalized double Pareto, the inverse gamma and half-$t$ priors are of the form \eqref{eq:Ghoshform} as well.

The global-local scale prior is of the form \eqref{eq:theprior} with
\begin{align*}
\pi(u) = \frac{K\tau^{2a}} {u^{1+a}} L\Big(\frac u{\tau^2}\Big).
\end{align*}
We assume that the polynomial decay in $u$ is at least of order $3/2,$ that is $a\geq \tfrac 12.$ In particular, the horseshoe lies directly at the boundary in this sense. Depending on $a,$ we allow for different values of $\tau.$ If $\tfrac 12 \leq a< 1,$ we assume $\tau^{2a}\leq (p_n/n)\sqrt{\log(n/p_n)};$ if $a=1,$ we assume $\tau^{2}\leq p_n/n;$ and if $a>1,$ we assume $\tau^{2}\leq (p_n/n)\log(n/p_n).$

Below, we check Conditions \ref{cond1}-\ref{cond3}.

{\it Condition 1':} It is enough to show that $\pi'$ is a uniformly regular varying function. Notice that $L$ is uniformly regular varying and satisfies \eqref{eq.unif_reg_vary_at_zero_prop} with $R=M/c_0$ and $z_0=t_0.$ If two functions are uniformly regular varying, then also their product, and thus $\pi'$ is uniformly regular varying.

{\it Condition 2:} Because of $p_n=o(n),$ $\tau^2 \rightarrow 0.$ Observe that $u \geq t_0\tau^2$ implies $L(u/\tau^2) \geq c_0$ and thus
\begin{align*}
\int_0^1 \pi(u) du &\geq \int_{t_0\tau^2}^{(t_0+1)\tau^2} \pi(u) du \geq  \int_{t_0\tau^2}^{(t_0+1)\tau^2} \frac{c_0K\tau^{2a}}{u^{1+a}}du = \frac{c_0 K}{(t_0 + 1)^{1+a}}.
\end{align*}

{\it Condition 3:}
Since $L$ is bounded in sup-norm by $M,$ and $s_n \geq \tau^2,$ we find that $\pi(u)\leq KM\tau^{2a}u^{-1-a} ,$ for all $u\geq s_n.$ With this bound, it is straightforward to verify Condition \ref{cond3}.

Thus, we can apply Theorem \ref{thm:main}. \qed

In particular, the posterior concentration theorem holds even more generally than shown by \cite{Ghosh2015}, as the restriction $a < 1$ can be removed. Thus, for example, we recover Theorem 3.3 of \cite{vanderPas2014} and in addition, find that  the normal-exponential-gamma prior of \cite{Griffin2005} contracts at at most the minimax rate for $\gamma = p_n/n$ and any $\lambda \geq 1/2$.

\subsection{The inverse-Gaussian prior}
Caron and Doucet \cite{Caron2008} propose to use the inverse-Gaussian distribution as prior for $\sigma^2.$ For positive constants $b$ and $\tau$ the variance $\sigma^2$ is drawn from an inverse Gaussian distribution with mean $\sqrt{2}\tau$ and shape parameter $\sqrt{2b}$. Thus the prior on the components is of the form \eqref{eq:theprior} with
\begin{equation*}
	\pi(u) = \frac{C_{b,\tau} \tau}{ u^{3/2}} e^{-\frac{ \tau^2}{u} - b u},
\end{equation*}
where $C_{b,\tau} =  e^{2\sqrt{b}\tau}/\sqrt{\pi}$ is the normalization factor. (In the notation of \cite{Caron2008}, this corresponds to reparametrizing $\gamma = \sqrt{2b},$ $\alpha/n = \sqrt{2}\tau,$ and $K=n$ is the dimension of the unknown mean vector.) As $\tau$ becomes small the distribution is concentrated near zero. \cite{Caron2008} suggests to take $\tau$ proportional to $1/n,$ and we find that optimal rates can be achieved if $(p_n/n)^K\lesssim \tau \leq (p_n/n) \sqrt{\log (n/p_n)}$ for some $K>1.$

Below we verify Condition \ref{cond1} and Condition \ref{condA}, which together imply Theorem \ref{thm:main}. The inverse-Gaussian prior does not fit within the class considered by \cite{Ghosh2015}, because of the additional exponential factors. 

{\it Condition 1:} For $u\geq 1,$ $e^{-1} \leq e^{-\tau^2/u}\leq 1.$ Thus, $u\mapsto e^{-\tau^2/u}$ is uniformly regular varying with constants $R = e$ and $z_0=1.$ Since products of uniformly regular varying functions are again uniformly regular varying, we can write $\pi(u) = L_n(u) e^{-bu}$ with $L_n$ uniformly regular varying. 

For $u\geq 1,$ $\pi(u) \geq \pi^{-1/2} e^{-1} \tau u^{-3/2} e^{-bu},$ using the explicit expression for the constant $C_{b,\tau}.$ Thus, \eqref{eq.assump_on_lb_Ln} holds with $b'>b,$ $K=\alpha,$ $z_*=1,$ and $C'$ a sufficiently large constant.

{\it Condition A:} Observe that $\pi(u)\leq C_{b,1} \tau u^{-3/2}.$ 

Hence, the statement of Theorem \ref{thm:main} follows. \qed

\subsection{The horseshoe+ prior}
The horseshoe+  prior was introduced by \cite{Bhadra2015}. It is an extension of the horseshoe including an additional latent variable. A Cauchy random variable with parameter $\lambda$ that is conditioned to be positive is said to be half-Cauchy and we write $C^+(0,\lambda)$ for its distribution. The horseshoe+ prior can be defined via the hierarchical construction
\begin{equation*}
\theta_i \mid \sigma_i \sim \mathcal{N}(0, \sigma_i^2), \quad \sigma_i \mid \eta_i, \tau \sim C^+(0, \tau\eta_i), \quad \eta_i \sim C^+(0,1).
\end{equation*}
 and should be compared to the horseshoe prior
\begin{equation*}
\theta_i \mid \sigma_i \sim \mathcal{N}(0, \sigma_i^2), \quad \sigma_i \mid  \tau \sim C^+(0, \tau).
\end{equation*}
The additional variable $\eta_i$ allows for another level of shrinkage, a role which falls solely to $\tau$ in the horseshoe prior.  In \cite{Bhadra2015}, the claim is made that the horseshoe+ is an improvement over the horseshoe in several senses, but no posterior concentration results are known so far. With Theorem \ref{thm:main}, we can show that the horseshoe+ enjoys the same upper bound  on the posterior contraction rate as the horseshoe, if $(p_n/n)^K\lesssim \tau \lesssim (p_n/n)(\log(n/p_n))^{-1/2},$ for some $K>1.$

The horseshoe+ prior is of the form \eqref{eq:theprior} with
\begin{equation*}
\pi(u) = \frac{\tau}{\pi^2 } { \log(u/\tau^2) \over (u - \tau^2) u^{1/2}}.
\end{equation*}
Below, we verify  Conditions \ref{cond1}-\ref{cond3}.

{\it Condition 1:} Write  $\pi(u) = L_n(u),$ that is, $b = 0.$ Let us show that $L_n$ is uniformly regular varying. For that define $u_0 := 2$. For $u > u_0$, and $\tau^2 \leq 1$ we have $ u/2 \leq u-\tau^2 \leq u$, thus 
\begin{equation*}
\frac{1}{2}a^{-3/2} \frac{\log(u/\tau^2) + \log(a) }{\log(u/\tau^2)} \leq \frac{\pi(au)}{\pi(u)} \leq 2 a^{-3/2} \frac{\log(u/\tau^2) + \log(a) }{\log(u/\tau^2)}  .
\end{equation*}
Since 
\begin{equation*}
1 \leq \frac{\log(u/\tau^2) + \log(a) }{\log(u/\tau^2)} \leq 2, 
\end{equation*}
$L_n$ is regular varying. To check the second part of the assumption, observe that $\pi(u) \geq \pi^{-1} \tau u^{-3/2} \log(u/\tau^2)$. For any $K>\alpha$ and any $b' > 0$,
\begin{equation*}
\pi(u) e^{b'u} \gtrsim \tau \log(1/\tau) \geq \Big(\frac{p_n}{n}\Big)^K,  \quad \text{for all} \ u\geq u_0.
\end{equation*}
Thus, Condition \ref{cond1} holds.

{\it Condition 2:} Observe that
\begin{align*}
\int_0^1 \pi(u) du &\geq \frac{\tau}{\pi^2} \int_0^{\tau^2/2}  {\log(\tau^2/u) \over (\tau^2 - u) u^{1/2}} du
\geq \frac{\tau}{\pi^2} \frac{1}{(\tau^2/2)^{3/2}}\cdot \frac{\tau^2}{2} \log\tfrac{1}{2}
 \gtrsim 1.
\end{align*}

{\it Condition 3:} For any $u\geq s_n$ we can use $(u-\tau^2) \geq u/2.$ This shows that 
\begin{align*}
	\pi(u) \leq \frac{\tau\log (u)}{u^{3/2}} + 	\frac{\tau\log (1/\tau^2)}{u^{3/2}}, \quad \text{for all} \ u\geq s_n.
\end{align*}
In particular, $\pi(u)\lesssim \tau \log(n/p_n)/u^{3/2}$ for $s_n\leq u\leq b_n^2.$ For the integral on $[b_n^2, \infty),$ we use that $\tfrac{d}{du}-(\log(u)+1)/u = \log(u)/u^2.$ Together, Condition \ref{cond3} follows thanks to $\tau \lesssim (p_n/n)/\sqrt{\log(n/p_n)}.$

Thus, Theorem \ref{thm:main} can be applied. \qed

\subsection{Normal-gamma prior}
\label{sec:ex:norm:gamma}
The normal-gamma prior, discussed by \cite{Caron2008} and \cite{Griffin2010}, takes the following form for shape parameter $\tau>0$ and rate parameter $\beta>0$:
\begin{equation*}
\pi(u) = \frac{\beta^\tau}{\Gamma(\tau)} u^{\tau - 1}e^{-\beta u} = \frac{\tau \beta^\tau}{\Gamma(\tau+1)} u^{\tau - 1}e^{-\beta u}.
\end{equation*}
In \cite{Griffin2010}, it is observed that decreasing $\tau$ leads to a distribution with a lot of mass near zero, while preserving heavy tails. This is also illustrated in the right-most panels of Figure \ref{fig:effecttau}. The class of normal-gamma priors includes the double exponential prior as a special case, with $\tau = 1$. We now show that the normal-gamma prior satisfies the conditions of Theorem \ref{thm:main} for any fixed $\beta$, and for any $(p_n/n)^K \lesssim \tau \lesssim (p_n/n)\sqrt{\log(n/p_n)} \leq 1$ for some fixed $K$.

Below, we check Conditions \ref{cond1}-\ref{cond3}.

{\it Condition 1:} We define $L_n(u) = \frac{\beta^\tau}{\Gamma(\tau)} u^{\tau-1}$, so $\pi(u) = L_n(u)e^{-bu}$ with $b = \beta$. Note that since $\tau \to 0$, we have that  there exist a constant $C$ such that $C^{-1} \leq \beta^\tau \leq C$. We now prove that $L_n$ is regular varying. We have 
 $$\frac{L_n(au)}{L_n(u)} = a^{\tau-1}.$$
 and thus for all $a \in [1,2]$, $a^{-1} \leq L_n(au)/L_n(u) \leq 1$. 
  In addition for $u > u_*:=1$ we have, using $\Gamma(\tau+1)\geq \Gamma(1)=1$, 
  $$L_n(u) = \frac{\tau\beta^\tau}{\Gamma(\tau+1)} u^{\tau-1} 
  \geq \frac{(\beta \wedge 1)\tau}{\Gamma(2) u}
  \gtrsim \Big( \frac{p_n}n \Big)^K \frac 1u, $$
 implying $\pi(u) = L_n(u)u^{-1}e^{-\beta u} \gtrsim (p_n/n)^K e^{-2\beta u}.$ Thus 
 Condition \ref{cond1} is satisfied. 

{\it Condition 2:} 
\begin{equation*}
\int_0^{1} \pi(u) du \geq \frac{(\beta \wedge 1)e^{-bu} \tau }{ \Gamma(2)}  \int_0^1 u^{\tau -1} du = \frac{ (\beta \wedge 1) e^{-bu}}{ \Gamma(2)} \gtrsim 1.
\end{equation*}

{\it Condition 3:} Notice that $\pi(u) \leq (\beta \vee 1) \tau u^{\tau -1},$ for all $u\leq 1.$ For $u\geq 1,$ we find $\pi(u) \leq (\beta \vee 1) \tau e^{-\beta u}$. Since $e^{-\beta u}$ decays faster than any polynomial power of $u,$ we see that Condition \ref{cond3} holds thanks to $b_n \tau \lesssim s_n.$

Thus, we can apply Theorem \ref{thm:main}.

In \cite{Griffin2010}, it is discussed that the extra modelling flexibility afforded by generalizing the double exponential prior to include the parameter $\tau$ is essential, and indeed the double exponential ($\tau = 1$) does not allow a dependence on $p_n$ and $n$ such that our conditions are met.

\subsection{Spike-and-slab Lasso prior}

The spike-and-slab Lasso prior was introduced by \cite{Rockova2015}. It may be viewed as a continuous version of the usual spike-and-slab prior with a Laplace slab, as studied in \cite{Castillo2012, Castillo2015}, where the spike component has been replaced by a very concentrated Laplace distribution. Recent theoretical results, including posterior concentration at the minimax rate, have been obtained in \cite{Rockova2015}. Here, we recover Corollary 6.1 of \cite{Rockova2015}.

For a fixed constant $a>0$ and a sequence $\tau \rightarrow 0,$ we define the spike-and-slab Lasso  as prior of the form \eqref{eq:theprior} with hyperprior
\begin{align}
	\pi(u) = \omega a e^{-a u } + (1-\omega) \frac 1{\tau} e^{- \frac{u}{\tau}}, \quad u>0
	\label{eq.pi_spike_and_slab_LASSO}
\end{align}
on the variance. Recall that the Laplace distribution with parameter $\lambda$ is a scale mixture of normals where the mixing density is exponential with parameter $\lambda^2.$ Applied to model \eqref{eq:theprior}, the prior on $\theta_i$ is thus a mixture of two Laplace distributions with parameter $\sqrt{a}$ and $\tau^{-1/2}$ and mixing weights $\omega$ and $1-\omega,$ respectively and this justifies the name.

We now prove that the prior satisfies the conditions of Theorem \ref{thm:main} for mixing weights satisfying $(p_n/n)^K\leq \omega \leq (p_n/n)\sqrt{\log(n/p_n)}\leq \tfrac 12,$ for some $K>1$ and $\tau = (p_n/n)^{\alpha}$ with $\alpha \geq 1$.

{\it Condition 1:} To prove that Condition 1 holds we rewrite the prior $\pi$ as 
$$
\pi(u) = e^{-au} \left( a \omega + \frac{1-\omega}{\tau} e^{-u(\frac{1}{\tau} - a)} \right) =: e^{-au} L_n(u)
$$ 
For $n$ large enough, we have $1/\tau -a>1/(2\tau)$. For all $u> 1$ and for $C>0$ a constant depending only on $K$ and $\alpha,$ 
\begin{align*}
\frac{1-\omega}{\tau} e^{-u(\frac 1{\tau} -a)} \leq\frac 1{\tau} e^{-\frac{1}{2\tau}} \leq C\tau^{\frac K{\alpha}} \leq C\omega.
\end{align*}
Hence, for sufficiently large $n$, $a \omega \leq  L_n(u) \leq (a+C) \omega$ for all $u\geq 1.$ Thus $L_n$ is regular varying with $u_0=1$. Since also $\pi(u) \geq a \omega e^{-au}$ and $\omega \geq (p_n/n)^K$, Condition \ref{cond1} holds. 

{\it Condition 2:} $\int_0^1 \pi(u) du \geq (1-\omega) \int_0^\tau \frac 1{\tau} e^{-\frac u{\tau}} du = (1-\omega) (1 - e^{-1}) .$

{\it Condition 3:} We might split the two mixing components in \eqref{eq.pi_spike_and_slab_LASSO} and write $\pi=:\pi_1+\pi_2.$ To verify the condition for the first component $\pi_1,$ we use that $e^{-au}\leq 1$ for $u\leq 1$ and that $e^{-au}$ decays faster than any polynomial for $u>1.$ In order that Condition \ref{cond3} is satisfied, we need thus $\omega \lesssim (p_n/n)\sqrt{\log (n/p_n)}.$ For $\pi_2,$ there exists a constant $C$ such that $\pi_2(u)\leq C\tau/u^2$ for all $u\geq s_n,$ due to $s_n\geq \tau.$ Straightforward computations show that $\pi_2$ satisfies Condition \ref{cond3} since $\tau \leq p_n/n.$

Thus, we can apply Theorem \ref{thm:main}. \qed

\section{\label{sec:simulation}Simulation results}

To illustrate the point that our conditions are very sharp, we compute the average square loss for two priors that do not meet our conditions, and compare them with two of the examples from Section \ref{sec:examples}. 

The first prior that does not meet the conditions is of the form \eqref{eq:Ghoshform} of Section \ref{subsec:ex:ghosh} with $a = 0.1 \leq 1/2$, $L(u) = e^{-1/u}$ and density,
$$
\pi_1(u) \propto u^{-1.1} e^{-\tau_1^2/u}, 
$$ 
and we take $\tau_1 = p_n/n$. Note that $\pi_1$ does not meet our conditions, as explained in Section \ref{subsec:ex:ghosh}, and will be called a \emph{bad} prior.  
The second prior included in this simulation that does not fit our assumptions is the Laplace prior  (see Section \ref{sec:ex:norm:gamma}). The two priors considered in this simulation study that do meet the conditions are the horseshoe and the normal-gamma priors, both with $\tau = p_n/n$. 

For each of these priors, we sample from the posterior distribution using a Gibbs Sampling algorithm, following the one proposed for the horseshoe prior by \cite{Carvalho2010}. To do so, we first compute the full conditional distributions
\begin{align*}
p(\beta | X, \sigma^2) &= \frac{1}{\sqrt{2 \pi \hat{\sigma}^2}} e^{-\frac{1}{2\hat{\sigma}^2} (\beta - \hat{\beta})^2} \\
p(\sigma^2|X,\beta) &\propto (\sigma^2)^{-1/2} e^{-\frac{\beta^2}{2 \sigma^2}} \pi(\sigma^2),
\end{align*}
where $ \hat{\sigma}^2 = \sigma^2/(1 + \sigma^2)$ and $\hat{\beta} = X\sigma^2 / (1 + \sigma^2)$. The only difficulty is thus sampling from $p(\sigma^2|X,\beta)$. For the horseshoe prior we follow the approach proposed by \cite{Carvalho2010}. We apply a similar method for the normal-gamma prior using the approach proposed by \cite{Damien1999Gibbs}. Sampling from the \emph{bad} prior is even simpler given that in this case $p(\sigma|X, \beta)$ is an inverse gamma. 
We compute the average square loss on $500$ replicates of simulated data of size $n = 100,200,500,1000$. For each $n$, we fix the number of nonzero means at $p_n = 10$, and take the nonzero coefficients equal $5\sqrt{2\log{n}}$. This value is well past the 'universal threshold' of $\sqrt{2\log{n}}$, and thus the signals should  be relatively easy to detect. For each data set, we compute the posterior square loss using $5000$ draws from the posterior with a burn-in of $20\%$. 

\begin{figure}
\includegraphics[width = 0.8\textwidth]{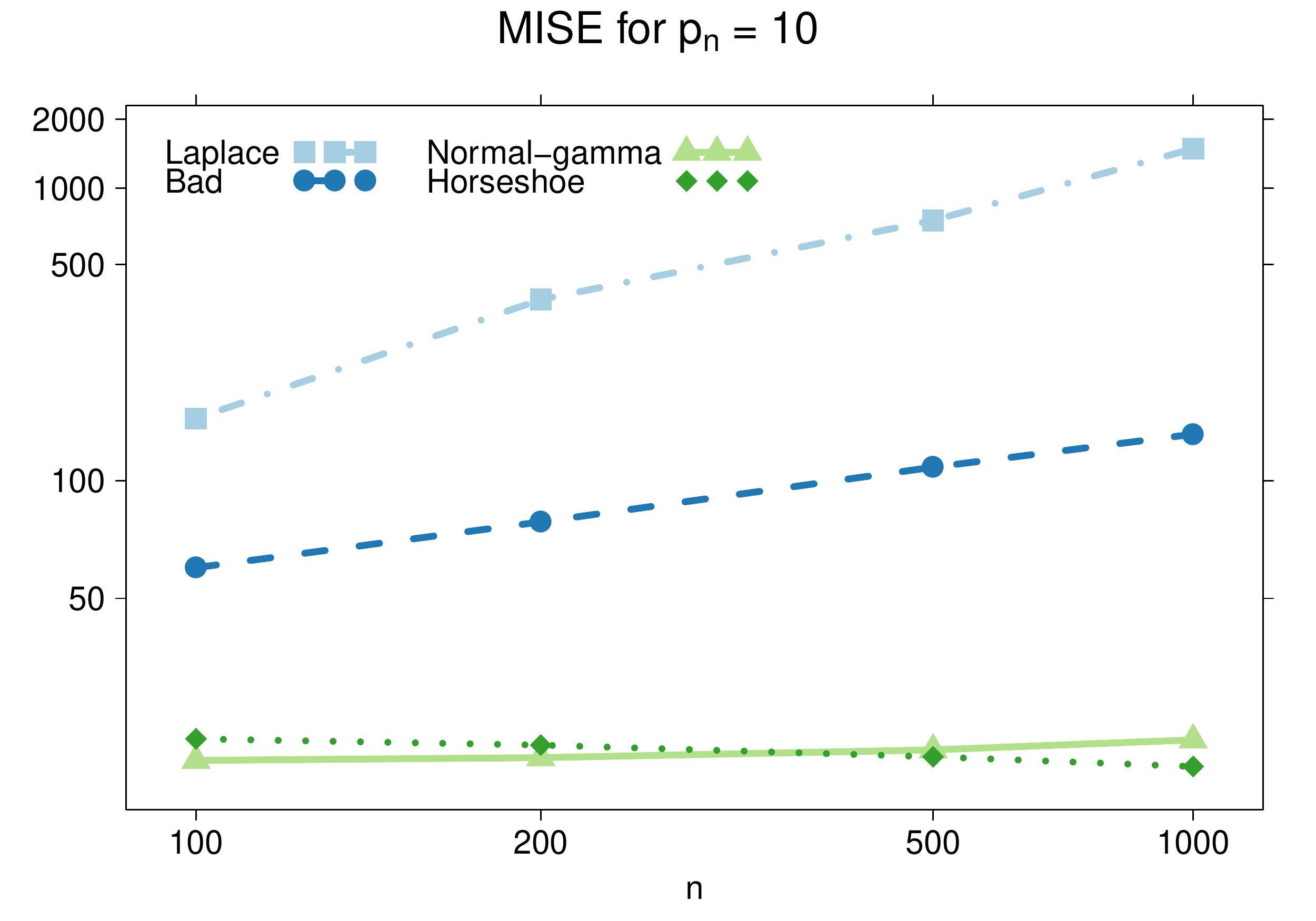}
\caption{The logarithm of the integrated square loss for the Laplace, \emph{bad}, Normal-Gamma and Horseshoe priors plotted against $\log\log{n}$, computed on 500 replicates of the data for each value of $n$. The axis labels refer to the original, non-log-transformed scale.}
\label{fig:Mise}
\end{figure} 

The results are presented in Figure \ref{fig:Mise}. Given that $p_n=10$ is fixed, if the posterior contracts at the minimax rate, then the integrated square loss should be linear in $\log{n}$. However, we see that for the Laplace and \emph{bad} priors, the slope of the loss grows with $n$, when it remains steady for the other two considered priors. This suggest that the horseshoe and normal-gamma have a risk of a lower order than the \emph{bad} and Laplace priors, illustrating that our conditions are very sharp.

\section{Discussion}
Our main theorem, Theorem \ref{thm:main}, expands the class of shrinkage priors with theoretical guarantees for the posterior contraction rate. Not only can it be used to obtain the optimal posterior contraction rate for the horseshoe+, the inverse-Gaussian and normal-gamma priors, but the conditions provide some characterization of properties of sparsity priors that lead to desirable behaviour. Essentially, the tails of the prior on the local variance should be at least as heavy as Laplace, but not too heavy, and there needs to be a sizable amount of mass around zero compared to the amount of mass in the tails, in particular when the underlying mean vector grows to be more sparse.

In \cite{Polson2010} global-local scale mixtures of normals like \eqref{eq:globallocal} are discussed, with a prior on the parameter $\tau^2$. Their guidelines are twofold: the prior on the local variance $\sigma_i^2$ should have heavy tails, while the prior on the global variance $\tau^2$ should have substantial mass around zero. They argue that any prior on $\sigma_i^2$ with an exponential tail will force a tradeoff between shrinking the noise towards zero and leaving the large nonzero means unshrunk, while the shrinkage of large signals will go to zero when a prior with a polynomial tail is chosen. This matches the intuition behind our conditions, with the remark that exponential tails \emph{are} possible, but they should not be lighter than Laplace.

Besides the three discussed goals of recovery, uncertainty quantification, and computational simplicity, we might have mentioned a fourth: performing \emph{model selection} or \emph{multiple testing}. Priors of the type studied in this paper are not directly applicable for this goal, as the posterior mean will, with probability one, not be exactly equal to zero. A model selection procedure can be constructed however, for example by thresholding using the observed values of $m_{x_i}$: if $m_{x_i}$ is larger than some constant, we consider the underlying parameter to be a signal, and otherwise we declare it noise. Such a procedure was proposed for the horseshoe by \cite{Carvalho2010}, and was shown to enjoy good theoretical properties by \cite{Datta2013}. Similar results were found for the horseshoe+ \cite{Bhadra2015}. The same thresholding procedure, and similar analysis methods, may prove to be fruitful for the more general prior \eqref{eq:theprior}.

\appendix

\section{\label{sec:proofs}Proofs}
This section contains the proofs of Theorem \ref{thm:main} and Theorem \ref{thm.necessary}, followed by the statement and proofs of the supporting Lemmas. The proof of Theorem \ref{thm:main} follows the same structure as that of Theorem 3.3 in \cite{vanderPas2014}, but requires more general methods to bound the integrals involved in the proof.   

In the course of the proofs, we use the following two transformations of $\pi,$
\begin{align}
g(z) &= \frac{1}{z^2}\pi\left(\frac{1-z}{z}\right) \quad \text{and} \quad
h(z) = \frac1{(1-z)^{3/2}} \pi\left(\frac{z}{1-z}\right). \label{eq:defh}
\end{align}
The function $g$ is a density on $[0,1]$, resulting from transforming the density $\pi$ on $\sigma_i^2$ to a density for $z = (1+\sigma_i^2)^{-1}$. The function $h$ is a rescaled version of $\pi$.\\

\begin{lem}\label{lem:cond1prime}
Condition \ref{cond1prime} implies Condition \ref{cond1}.
\end{lem}

\begin{proof}
Observe that $\pi(u) = \widetilde \pi(u/\tau^2)/\tau^2.$ Since by assumption $\widetilde \pi$ is uniformly regular varying, \eqref{eq.unif_reg_vary_at_zero_prop} holds for some constants $R$ and $u_0$ which do not depend on $n.$ To check the first part of Condition \ref{cond1}, it is enough to see that $\widetilde \pi(\cdot /\tau^2)$ is uniformly regular varying as well and satisfies \eqref{eq.unif_reg_vary_at_zero_prop} with the same constants as $\widetilde \pi.$

It remains to prove a lower bound \eqref{eq.assump_on_lb_Ln}. Thanks to $\tau^2 \leq 1$ and Lemma \ref{lem.L(az)_bd}, for any $u\geq u_*:=u_0,$ $\widetilde \pi(u/\tau^2) \geq \widetilde \pi(u_0) (\tau^2u_0/2u)^{\log_2 R}.$ This implies the lower bound \eqref{eq.assump_on_lb_Ln} with $K=2\alpha \log_2R,$ $b'>0,$ and $C'$ a sufficiently large constant. 
\end{proof}

\begin{proof}[Proof of Theorem \ref{thm:main}]
Applying Lemma \ref{lem.non_zero_components} gives under Condition \ref{cond1}, $\sum_{i: \theta_i \neq 0} \E_{\theta_i} (\theta_i - \widehat\theta_i)^2\lesssim p_n\log(n/p_n)$ and  $\sum_{i:\theta_i \neq 0} \E_{\theta_i} \Var( \theta_i \mid X_i) \lesssim p_n\log(n/p_n).$ These inequalities combined with Markov's inequality prove the first two statements of the theorem. Similarly, under Condition  \ref{cond2} and Condition \ref{cond3}, we obtain from Lemma  \ref{lem:bias.zero} and Lemma \ref{lem:var.zero}, $\E_\theta \sum_{i: \theta_i =0}\widehat \theta_i^2 \leq n\E_0(Xm_X)^2\lesssim p_n\log(n/p_n)$ and $\sum_{i:\theta_i = 0} \E_0 \Var( \theta_i \mid X_i)\lesssim p_n\log(n/p_n) .$ Together with Markov's inequality, this proves the third and fourth statement of the theorem.
\end{proof}

\begin{proof}[Proof of Theorem \ref{thm.necessary}]
Without loss of generality, we can take $\kappa_n$ such that $\kappa_n\geq n^{-1/4}$ for all $n.$ Consider the prior, where $\theta_i$ is drawn from the Laplace density with parameter $\lambda =\sqrt{s_n/\kappa_n}.$ This prior is of the form \eqref{eq:theprior} with $\pi(u)=\lambda^2 e^{-\lambda^2 u}$ (cf. Section \ref{sec.necessary}). Theorem 7 in \cite{Castillo2015} shows that \eqref{eq.thm_nec_concl} holds with $M_n =1/\kappa_n\rightarrow \infty.$ Thus it remains to prove that $\pi$ satisfies Condition \ref{cond2} and Condition \ref{cond3modified}.

Condition \ref{cond2} follows immediately. For Condition \ref{cond3modified} observe that due to $\kappa_n\geq n^{-1/4},$  $\lambda \geq n^{1/4}/\sqrt{\log n}.$ Splitting the integral $\int_0^{\lambda^2}=\int_0^1 +\int_1^{\lambda^2},$ we find $\kappa_n\int_{s_n}^1u\pi(u) du \leq \kappa_n\int_0^1 u\lambda^2 e^{-\lambda^2 u}du \leq \kappa_n\lambda^{-2}\int_0^{\lambda^2} ve^{-v}dv \lesssim \kappa_n\lambda^{-2}= s_n.$ Also, $\int_1^{b_n^2} u\pi(u) du =\lambda^{-2}\int_{\lambda^2}^{b_n^2\lambda^2} ve^{-v}dv\leq b_n^2 e^{-\lambda^2} =o(s_n)$ and $b_n^3 \int_1^\infty \pi(u)/\sqrt{u} du\leq b_n^3 \int_1^\infty \pi(u) du \leq b_n^{3}e^{-\lambda^2}=o(s_n).$ Hence, Condition \ref{cond3modified} holds and this completes the proof.
\end{proof}

\begin{lem}\label{lem:expr.var}
The posterior variance can be written as
\begin{equation}\label{eq:var2}
\Var(\theta \mid x) = m_x - (xm_x-x)^2 + x^2\frac{\int_0^1 (1-z)^2 h(z) e^{\frac{x^2}{2}z}dz}{\int_0^1  h(z) e^{\frac{x^2}{2}z}dz}
\end{equation}
and bounded by
\begin{align}
\Var(\theta \mid x) &\leq 1 
+ x^2\frac{\int_0^1 (1-z)^2 h(z) e^{\frac{x^2}{2}z}dz}{\int_0^1  h(z) e^{\frac{x^2}{2}z}dz}
\quad \text{and} \quad
\Var(\theta \mid x) \leq m_x + x^2m_x. \label{eq.var_ub}
\end{align}
\end{lem}

\begin{proof}
By Tweedie's formula \cite{Robbins1956}, the posterior variance for  $\theta_i$ given an observation $x_i$ is equal to $1 + (d^2/dx^2) \log p(x)|_{x = x_i}$, where $p(x_i)$ is the marginal distribution of $x_i$. Computing
\begin{align*}
p(x)  &= \int_0^1 \frac{1}{\sqrt{2\pi}} (1-z)^{-3/2}e^{-\frac{x^2}{2}(1-z)} \pi\left(\frac{z}{1-z}\right) dz,
\end{align*}
taking derivatives with respect to $x$, and substituting $h(z) = (1-z)^{-3/2}\pi(z/(1-z))$ gives 
\begin{equation*}
 \Var(\theta \mid x)= 1 
+ x^2\frac{\int_0^1 (1-z)^2 h(z) e^{\frac{x^2}{2}z}dz}{\int_0^1  h(z) e^{\frac{x^2}{2}z}dz} 
- \frac{\int_0^1 (1-z) h(z) e^{\frac{x^2}{2}z}dz}{\int_0^1 h(z) e^{\frac{x^2}{2}z}dz}
- x^2\left(\frac{\int_0^1 (1-z) h(z) e^{\frac{x^2}{2}z}dz}{\int_0^1 h(z) e^{\frac{x^2}{2}z}dz} \right)^2.
\end{equation*}
From that we can derive \eqref{eq:var2} noting that the third term on the r.h.s. is $1 - m_x.$ The last display also implies the first inequality in \eqref{eq.var_ub}. Representation \eqref{eq:var2} together with the trivial bound $(1-z)^2 \leq (1-z)$ for $z \in [0,1]$ yields
\begin{equation*}
x^2\frac{\int_0^1 (1-z)^2 h(z) e^{\frac{x^2}{2}z}dz}{\int_0^1  h(z) e^{\frac{x^2}{2}z}dz}
\leq x^2\frac{\int_0^1 (1-z) h(z) e^{\frac{x^2}{2}z}dz}{\int_0^1  h(z) e^{\frac{x^2}{2}z}dz} = x^2(1-m_x).
\end{equation*}
Combined with \eqref{eq:var2}, we find $\Var(\theta \mid x) \leq m_x - x^2m_x^2 + x^2m_x \leq m_x + x^2m_x.$
\end{proof}

\begin{lem}
\label{lem.L(az)_bd}
Suppose that $L$ is uniformly regular varying. If $R$ and $u_0$ are chosen such that \eqref{eq.unif_reg_vary_at_zero_prop} holds, then, for any $a\geq 1,$
\begin{align*}
	L(u) \leq (2a)^{\log_2 R} L(au),
\end{align*}
where $\log_2$ denotes the binary logarithm.
\end{lem}

\begin{proof}
Write $a=2^rb$ with $r$ a non-negative integer and $1\leq b<2.$ By assumption \eqref{eq.unif_reg_vary_at_zero_prop} holds for some $R$ and $u_0.$ We apply the upper bound \eqref{eq.unif_reg_vary_at_zero_prop} repeatedly and obtain for $a\geq 1,$  $L(u)\leq R L(2u) \leq \ldots \leq R^r L(2^r u) \leq R^{r+1}L(au).$ Since $R^{r+1}=(2^{r+1})^{\log_2 R} \leq (2a)^{\log_2 R},$ the result follows.
\end{proof}

\begin{lem}
\label{lem.Ln_shift}
Assume that $L$ is uniformly regular varying and satisfies \eqref{eq.unif_reg_vary_at_zero_prop} with $R$ and $u_0.$ Then,  the shifted function $L(\cdot -1)$ is also uniformly regular varying with constants $R^3$ and $u_0\vee 2.$
\end{lem}

\begin{proof}
Write
\begin{align*}
	\frac{L(az-1)}{L(z-1)} = \frac{L(az-1)}{L(az)} \cdot \frac{L(az)}{L(z)} \cdot \frac{L(z)}{L(z-1)}.
\end{align*}
For $z\geq z_0\vee 2$ we apply \eqref{eq.unif_reg_vary_at_zero_prop} to each of the three fractions and this completes the proof.
\end{proof}

The following lemma states that if the density $g$ can be decomposed as a product of a function that is uniformly varying and possibly $n$ dependent, and a factor of the form $z\mapsto e^{-bz},$ then the posterior recovers the size of the non-zero components of $\theta$ with the minimax estimation rate, provided that the $n$ dependence is of the right order. \\

\begin{lem}
\label{lem.non_zero_components}
If Condition \ref{cond1} holds, there exists a constant $C,$ which is independent of $n$, such that 
\begin{align}\label{eq:non.zero.bias}
	\sum_{i: \theta_i \neq 0} \E_{\theta_i} (X_im_{X_i} -\theta_i)^2 \leq C p_n \log (en/p_n),
\end{align}
and
\begin{align}\label{eq:non.zero.var}
	\sum_{i: \theta_i \neq 0} \Var (\theta_i| X_i) \leq C p_n \log (en/p_n).
\end{align}
\end{lem}

\begin{proof}
We prove the two statements separately. The main argument is a careful analysis of the integral representation
\begin{align*}
	|x(m_x-1)| 
	= |x| \frac{\int_0^1  e^{ -\frac{x^2}{2}z} z^{-1/2} \pi\big(\tfrac 1z -1\big) dz}{\int_0^1  e^{-\frac{x^2}{2}z}z^{-3/2} \pi\big(\tfrac 1z -1\big) dz } 
	= |x| \frac{\int_0^1  e^{ -\frac{x^2}{2}u} u^{3/2} g(u) du}{\int_0^1  e^{-\frac{x^2}{2}u}u^{1/2}g(u) du }
\end{align*}
(cf. \eqref{eq:defmx} and \eqref{eq:defh}). Throughout the remaining proof, let $C_1$ be a generic constant which is independent of $n$ and which might change from line to line. Without loss of generality, we may assume that $u_0\geq 2.$

{\it Proof of \eqref{eq:non.zero.bias}:} It is enough to show $\sup_{x>0} |x(m_x-1)| \lesssim 1+\sqrt{\log (n/p_n)}.$ It is thus enough to consider the $\sup$ over $|x|>T_0:=2+2(u_0\vee u_*)+\sqrt{8u_0^{-1}K\log (n/p_n)},$ since otherwise, we simply use $|x(m_x-1)|\leq |x| .$

For $0\leq a<b\leq 1,$ write $I(a,b) = \int_a^b  e^{ -\frac{x^2}{2}u} u^{3/2} g(u) du / \int_0^1  e^{-\frac{x^2}{2}u}u^{1/2}g(u) du$ and for $b\leq a,$ set $I(a,b)=0.$ We need to prove that 
\begin{equation*}
I(0,1) = I\big(0,\tfrac{2b+4}{|x|}\big)+I\big(\tfrac{2b+4}{|x|}, u_0\big)+I\big(u_0,1\big) =: (I)+(II)+(III) \lesssim \frac 1{|x|}.
\end{equation*}
{\it Bound for $(I):$} Obviously, $I(0,v) \leq v$ for all $v\in (0,1].$ Thus, $I\big(0,\tfrac{2b+4}{|x|}\big)\leq C_1/|x|.$

{\it Bound for $(II):$} We first derive a lower bound for the denominator. Recall that by Condition \ref{cond1}, $\pi(u) = L_n(u) e^{-bu}.$ Define $\widetilde L_n = L_n(\cdot -1)$ and observe that due to $|x| \geq 2u_0$ we can use Lemma \ref{lem.Ln_shift} and substitute $v= u|x|/2$ to obtain
\begin{align}
	\int_0^1  e^{-\frac{x^2}{2}u}u^{-3/2} \pi\big(\tfrac 1u -1\big) du 
	&\geq 
	 \int_{1/|x|}^{2/|x|} e^{-\frac{x^2}{2}u}u^{-3/2}\widetilde L_n \big(\tfrac 1u\big)e^{-\frac bu+b} du  \\
	&\geq e^{b-(1+b)|x|} |x|^{3/2}  \int_{1/|x|}^{2/|x|} \widetilde L_n \big(\tfrac 1u\big) du \notag \\
	&= e^{b-(1+b)|x|} |x|^{1/2} 2 \int_{1/2}^1 \widetilde L_n \big(\tfrac 1v \cdot \tfrac {|x|} 2\big) dv \notag \\
	&\geq \frac 1{R^3}  e^{b-(1+b)|x|} |x|^{1/2}  \widetilde L_n\big(\tfrac {|x|} 2\big).
	\label{eq.denominator_non_zero_bd}
\end{align}
For the upper bound, using Lemma \ref{lem.L(az)_bd} with $u=|x|/v$ and $a=v/2,$
\begin{align*}
	&\int_{(2b+4)/|x|}^{u_0^{-1}}   e^{ -\frac{x^2}{2}u} u^{-1/2} \pi\big(\tfrac 1u -1\big) du\\
	&= \sum_{k=1}^\infty \int_{(2b+4+k-1)/|x|}^{(2b+4+k)/|x|}  e^{-\frac{x^2}{2}u}u^{-1/2}\widetilde L_n\big(\tfrac 1u\big)e^{b-\frac bu} \mathbf{1}(u\leq u_0^{-1}) du \\
	&\leq e^b\sum_{k=1}^\infty e^{- \frac{|x|}2 (2b+4+k-1)} 
	\Big(\frac{|x|}{2b+4+k-1}\Big)^{1/2}
	\int_{(2b+4+k-1)/|x|}^{(2b+4+k)/|x|}  \widetilde L_n\big(\tfrac 1u\big) \mathbf{1}(u\leq u_0^{-1}) du\\
		&\leq e^b\sum_{k=1}^\infty e^{- \frac{|x|}2 (2b+2+k)} |x|^{-1/2}
	\int_{2b+4+k-1}^{2b+4+k}  \widetilde L_n\big(\tfrac {|x|}v\big)\mathbf{1}\big(v\leq \tfrac {|x|}{u_0}\big) dv\\
		&\leq e^{-|x|(b+1)} |x|^{-1/2}   \widetilde L_n\big(\tfrac {|x|}2\big)  e^b\sum_{k=1}^\infty e^{- \frac{|x|}2 k} 
	(2b+4+k)^{3\log_2 R}.
\end{align*}
Notice that the sum $\sum_{k=1}^\infty e^{- \frac{|x|}2 k} (2b+4+k)^{3\log_2 R}$ is bounded for $|x|>T_0.$ Since by assumption, $R$ does not depend on $n$, we find $I\big(\tfrac{2b+4}{|x|}, u_0\big)\leq C_1 /|x|.$

{\it Bound for $(III):$} In this case, we use that $g$ is a density and find $\int_{u_0}^1  e^{-\frac{x^2}{2}u}u^{3/2}g(u) du \leq e^{-u_0 x^2/2}.$ For the denominator, we find using \eqref{eq.denominator_non_zero_bd}, $|x|\geq 2+2u_*,$ and assumption \eqref{eq.assump_on_lb_Ln} 
\begin{align*}
	\int_0^1  e^{-\frac{x^2}{2}u}u^{-3/2} \pi\big(\tfrac 1u -1\big) du \geq \frac{1}{R^3} e^{-(1+\frac b 2 )|x|} |x|^{1/2} \pi \big(\tfrac{|x|}2 -1\big) \geq \frac{1}{R^3C'} \big(\tfrac {p_n}n \big)^K e^{-(1+b+b' )|x|} |x|^{1/2} .
\end{align*}
Combining this with the upper bound  and $(1+b+b')|x| \leq (1+b+b')^2/u_0 +u_0 x^2/4$ gives
\begin{align*}
	I\big(u_0,1\big) \leq C' R^3 \big(\tfrac n{p_n}\big)^K |x|^{-1/2} e^{(1+b+b')^2/u_0}e^{-u_0 x^2/4 }.
\end{align*}
Using that $x\mapsto |x|^{1/2}e^{-u_0 x^2/8 }$ is bounded, we find for $|x|>T_0,$ that $I\big(u_0,1\big) \leq C_1/|x|.$

The result for \eqref{eq:non.zero.bias} follows by combining the bounds $(I)-(III).$

{\it Proof of \eqref{eq:non.zero.var}:} Recall that \eqref{eq.var_ub} uses $h(u)= (1-u)^{-3/2}\pi(u/(1-u)).$ With \eqref{eq:defh}, $h(1-u) = u^{-3/2}\pi((1-u)/u) = u^{1/2} g(u).$ Therefore, we find
\begin{align*}
	\Var (\theta | x) \leq 1+ x^2 \frac{\int_0^1  e^{ -\frac{x^2}{2}u} u^{5/2} g(u) du}{\int_0^1  e^{-\frac{x^2}{2}u}u^{1/2}g(u) du }.
\end{align*}
Arguing as for \eqref{eq:non.zero.bias} completes the proof.
\end{proof}

Next, we provide the technical lemmas establishing the rate for the zero coefficients. Recall that $s_n = (p_n/n) \log(n/p_n)$  and define
\begin{align}
	q_n:=\frac{p_n} n \sqrt{\log(n/p_n)}.
	\label{eq.qn_def}
\end{align}
Suppose that Condition \ref{cond2} and Condition \ref{cond3} hold with constants $c$ and $C,$ respectively. With \eqref{eq:defmx},
\begin{align}
m_x 
:=  \frac{\int_0^\infty \tfrac{u} {(1+u)^{3/2}} e^{ \frac{x^2u}{2+2u}} \pi(u) du}{\int_0^\infty \tfrac 1{(1+u)^{1/2}} e^{ \frac{x^2u}{2+2u}} \pi(u) du}
&\leq s_n + \frac{\sqrt{2}}{c} \int_{s_n}^\infty \frac{u e^{ \frac{x^2u}{2+2u}}}{ (1+u)^{3/2}}  \pi(u) du \notag \\
&\leq s_n \big(1+ \frac{\sqrt{2}C}{c} e^{\frac{x^2}{4}}\big)+ \frac{\sqrt{2}}{c} \int_1^\infty \frac{u e^{ \frac{x^2u}{2+2u}}}{ (1+u)^{3/2}}  \pi(u) du \notag \\
&\leq s_n \big(1+ \frac{\sqrt{2}C}{c} e^{\frac{x^2}{4}}\big) +  \frac{\sqrt{8}C}{c} q_n e^{\frac{x^2}{2}},
\label{eq.mx_ineqs}
\end{align}
where for the last inequality, we split the integral $\int_1^\infty = \int_{1}^{\log(n/p_n)}+\int_{\log(n/p_n)}^\infty$ and used Condition \ref{cond3} twice.
These inequality will be very useful for the proofs below. For the variance bound, the last bound is not sharp enough and we need to work with the upper bound induced by the second inequality.\\

\begin{lem}\label{lem:bias.zero}
Work under Condition \ref{cond2} and Condition \ref{cond3}. Then,
\begin{equation*}
\E_0(Xm_X)^2 \lesssim   \frac{p_n}{n} \log(n/p_n).
\end{equation*}
\end{lem}

\begin{proof}
Let $q_n$ be as in \eqref{eq.qn_def} and set $a_n := \sqrt{2\log(1/q_n)}.$ Decompose
\begin{equation*}
\E_0(Xm_X)^2 = \E_0(Xm_X)^2\1\{|X| \leq a_n\} + \E_0(Xm_X)^2\1\{|X| > a_n\} =: I_1+I_2.
\end{equation*}
To bound the term $I_1,$ \eqref{eq.mx_ineqs} and $x^2e^{x^2/2}\leq \tfrac{d}{dx}[xe^{x^2/2}]$ yield $$I_1
 \lesssim s_n^2 \int_{-a_n}^{a_n} x^2 dx+ q_n^2 \int_{-a_n}^{a_n} x^2 e^{x^2/2} dx \lesssim s_n^2a_n^3+ q_n^2 a_ne^{a_n^2/2}.$$  There is a constant only depending on $K$ such that $x^2\log^K(1/x)\leq C_Kx$ for all $x\leq 1.$ Thus, $I_1\lesssim (p_n/n)\log(n/p_n).$

In order to bound $I_2,$ we use $m_x \leq 1$, $\tfrac{d}{dx}[-xe^{-x^2/2}] = -e^{-x^2/2} + x^2e^{-x^2/2}$ and Mills' ratio,
\begin{align*}
I_2 &\leq \E_0X^2\1\{|X| > a_n\} = 2\int_{a_n}^\infty x^2 \phi(x) dx
= 2[-x\phi(x)]_{a_n}^\infty +\int_{a_n}^\infty \phi(x) dx \leq e^{-a_n^2/2}(2a_n + 1).
\end{align*}
Plugging the expression for $a_n$ into the r.h.s. shows that also $I_2\lesssim (p_n/n)\log(n/p_n)$ and this finally gives $\E_0(Xm_X)^2 \lesssim   (p_n/n) \log(n/p_n).$
\end{proof}

\begin{lem}\label{lem:var.zero}
Work under Conditions \ref{cond2} and \ref{cond3}. Then,
\begin{equation*}
\sum_{i: \theta_i = 0}^n \E_0 \Var(\theta_i \mid X_i) \lesssim  p_n  \log(n/p_n).
\end{equation*}
\end{lem}

\begin{proof}
Let $a_n = \sqrt{2\log(n/p_n)}.$ It is enough to show that $\E_0\Var(\theta \mid X) \lesssim  p_n  \log(n/p_n) /n.$ To prove this, we need to treat the cases that $|X|$ is larger/smaller than $a_n$, separately. To bound the variance, we use \eqref{eq.var_ub}, that is $\Var(\theta \mid X)\leq m_x+x^2m_x\leq 1+x^2.$

{\it Case $|X|>a_n:$} Using the identity $d/dx[x\phi(x)] = \phi(x) - x^2\phi(x),$
\begin{align}
\E_0 \Var(\theta \mid X)\1_{\{|X| > a_n\}} &\leq 2\int_{a_n}^\infty (1+x^2) \phi(x)dx
= 2\Phi^c(a_n)  + 2\int_{a_n}^\infty x^2 \phi(x)dx\notag\\
&= 4\Phi^c(a_n)+ 2[-x\phi(x)]_{a_n}^\infty 
\leq 4\phi(a_n) + 2a_n\phi(a_n). \label{eq:var.zero.bound1}
\end{align}
Using the expression for $a_n$ shows that this can be further bounded by $(p_n/n)\sqrt{\log(n/p_n)}.$

{\it Case $|X|\leq a_n:$}  Notice that the variance bound implies $\Var(\theta \mid X)\leq m_x\1\{|x|\leq 1\}+2x^2m_x.$ Below, we estimate $\E_0 m_X\1\{|X| \leq 1\}$ and $\E_0 X^2 m_X\1\{|X| \leq a_n\}$. For the first term, using  \eqref{eq.mx_ineqs},
\begin{equation}\label{eq:var.zero.first.part}
\E_0 m_X\1\{|X| \leq 1\} \lesssim \int_{-1}^{1}( s_n e^{x^2/4}+q_n e^{x^2/2}) \phi(x) dx \leq 4s_n.
\end{equation} 
For the second term $\E_0 X^2m_X\1\{|X| \leq a_n\}$, we use the second inequality in \eqref{eq.mx_ineqs} and find
\begin{equation*}
\E_0 X^2m_X\1\{|X| \leq a_n\} \lesssim 
s_n  \int_{-a_n}^{a_n}x^2 e^{\frac{x^2}{4}}\phi(x)dx
+
\int_{-a_n}^{a_n} \int_1^\infty\frac{u\pi(u)}{(1+u)^{3/2}} x^2 e^{-\frac{x^2}{2+2u}} du dx.
\end{equation*}
The first integral is bounded by a constant and for the second integral, we use Fubini's theorem, substitute $y=x/\sqrt{1+u},$ and use Condition \ref{cond3}
\begin{align*}
	\int_{-a_n}^{a_n} \int_1^\infty \frac{u\pi(u)}{(1+u)^{3/2}} x^2 e^{-\frac{x^2}{2+2u}} du dx
	&=\int_1^\infty  u \pi(u) \int_{-a_n/\sqrt{1+u}}^{a_n/\sqrt{1+u}} y^2 e^{-\frac{y^2}2} dy du \\
	&\leq  \int_1^\infty  u \pi(u)  \Big[ \Big(\frac{a_n}{\sqrt{1+u}}\Big)^3\wedge \sqrt{2\pi} \Big]  du \\
	&\leq 2^{3/2}C s_n  .
\end{align*}
Together with \eqref{eq:var.zero.first.part} this shows that $\E_0 \Var(\theta \mid X)\1\{|X| \leq a_n\} \lesssim s_n.$ Since in both cases the upper bound is of order $(p_n/n)\log(n/p_n)$ the result follows.
\end{proof}

\bibliographystyle{acm}
\bibliography{sparseequidae}

\begin{thebibliography}{10}

\bibitem{andrews1974}
{\sc Andrews, D.~F., and Mallows, C.~L.}
\newblock Scale mixtures of normal distributions.
\newblock {\em J. R. Stat. Soc. Ser. B Stat. Methodol.\/} (1974), 99--102.

\bibitem{Bhadra2015}
{\sc Bhadra, A., Datta, J., Polson, N.~G., and Willard, B.}
\newblock The horseshoe+ estimator of ultra-sparse signals.
\newblock arXiv:1502.00560v2, 2015.

\bibitem{Bhattacharya2014}
{\sc Bhattacharya, A., Pati, D., Pillai, N.~S., and Dunson, D.~B.}
\newblock Dirichlet-{L}aplace priors for optimal shrinkage.
\newblock arXiv:1401.5398, 2014.

\bibitem{Caron2008}
{\sc Caron, F., and Doucet, A.}
\newblock Sparse {B}ayesian nonparametric regression.
\newblock In {\em Proceedings of the 25th International Conference on Machine
  Learning\/} (New York, NY, USA, 2008), ICML '08, ACM, pp.~88--95.

\bibitem{Carvalho2010}
{\sc Carvalho, C.~M., Polson, N.~G., and Scott, J.~G.}
\newblock The horseshoe estimator for sparse signals.
\newblock {\em Biometrika 97}, 2 (2010), 465--480.

\bibitem{Castillo2015}
{\sc Castillo, I., Schmidt-Hieber, J., and van~der Vaart, A.}
\newblock Bayesian linear regression with sparse priors.
\newblock {\em Ann. Statist. 43}, 5 (10 2015), 1986--2018.

\bibitem{Castillo2012}
{\sc Castillo, I., and Van~der Vaart, A.~W.}
\newblock Needles and straw in a haystack: Posterior concentration for possibly
  sparse sequences.
\newblock {\em Ann. Statist. 40}, 4 (2012), 2069--2101.

\bibitem{Damien1999Gibbs}
{\sc Damien, P., Wakefield, J., and Walker, S.}
\newblock Gibbs sampling for {B}ayesian non-conjugate and hierarchical models
  by using auxiliary variables.
\newblock {\em J. R. Stat. Soc. Ser. B Stat. Methodol. 61}, 2 (1999), 331--344.

\bibitem{Datta2013}
{\sc Datta, J., and Ghosh, J.~K.}
\newblock Asymptotic properties of {B}ayes risk for the horseshoe prior.
\newblock {\em Bayesian Analysis 8}, 1 (2013), 111--132.

\bibitem{Donoho1992}
{\sc Donoho, D.~L., Johnstone, I.~M., Hoch, J.~C., and Stern, A.~S.}
\newblock Maximum entropy and the nearly black object (with discussion).
\newblock {\em J. R. Stat. Soc. Ser. B Stat. Methodol. 54}, 1 (1992), 41--81.

\bibitem{Ghosal2000}
{\sc Ghosal, S., Ghosh, J.~K., and Van~der Vaart, A.~W.}
\newblock Convergence rates of posterior distributions.
\newblock {\em Ann. Statist. 28}, 2 (2000), 500--531.

\bibitem{Ghosh2015}
{\sc Ghosh, P., and Chakrabarti, A.}
\newblock Posterior concentration properties of a general class of shrinkage
  estimators around nearly black vectors.
\newblock arXiv:1412.8161v2, 2015.

\bibitem{Griffin2005}
{\sc Griffin, J.~E., and Brown, P.~J.}
\newblock Alternative prior distributions for variable selection with very many
  more variables than observations.
\newblock {\em Technical Report, University of Warwick.\/} (2005).

\bibitem{Griffin2010}
{\sc Griffin, J.~E., and Brown, P.~J.}
\newblock Inference with normal-gamma prior distributions in regression
  problems.
\newblock {\em Bayesian Analysis 5}, 1 (2010), 171--188.

\bibitem{Johnson2010}
{\sc Johnson, V.~E., and Rossell, D.}
\newblock On the use of non-local prior densities in {B}ayesian hypothesis
  tests.
\newblock {\em J. R. Stat. Soc. Ser. B Stat. Methodol. 72}, 2 (2010), 143--170.

\bibitem{Johnstone2004}
{\sc Johnstone, I.~M., and Silverman, B.~W.}
\newblock Needles and straw in haystacks: Empirical {B}ayes estimates of
  possibly sparse sequences.
\newblock {\em Ann. Statist. 32}, 4 (2004), 1594--1649.

\bibitem{Martin2014}
{\sc Martin, R., and Walker, S.~G.}
\newblock Asymptotically minimax empirical {B}ayes estimation of a sparse
  normal mean vector.
\newblock {\em Electron. J. Stat. 8}, 2 (2014), 2188--2206.

\bibitem{park2008}
{\sc Park, T., and Casella, G.}
\newblock {The Bayesian lasso}.
\newblock {\em J. Amer. Statist. Assoc. 103}, 482 (2008), 681--686.

\bibitem{Polson2010}
{\sc Polson, N.~G., and Scott, J.~G.}
\newblock Shrink globally, act locally: Sparse {B}ayesian regularization and
  prediction.
\newblock {\em Bayesian Statistics 9\/} (2010), 501--538.

\bibitem{Polson2012}
{\sc Polson, N.~G., and Scott, J.~G.}
\newblock Good, great or lucky? {S}creening for firms with sustained superior
  performance using heavy-tailed priors.
\newblock {\em Ann. Appl. Stat. 6}, 1 (2012), 161--185.

\bibitem{Polson2012-2}
{\sc Polson, N.~G., and Scott, J.~G.}
\newblock On the half-{C}auchy prior for a global scale parameter.
\newblock {\em Bayesian Analysis 7}, 4 (2012), 887--902.

\bibitem{Robbins1956}
{\sc Robbins, H.}
\newblock An empirical {B}ayes approach to statistics.
\newblock In {\em Proceedings of the Third Berkeley Symposium on Mathematical
  Statistics and Probability, Volume 1: Contributions to the Theory of
  Statistics\/} (Berkeley, California, 1956), University of California Press,
  pp.~157--163.

\bibitem{Rockova2015}
{\sc Ro\u{c}kov\'a, V.}
\newblock Bayesian estimation of sparse signals with a continuous
  spike-and-slab prior.
\newblock submitted manuscript, available at
  \url{http://stat.wharton.upenn.edu/~vrockova/rockova2015.pdf}, 2015.

\bibitem{Tibshirani1996}
{\sc Tibshirani, R.}
\newblock Regression shrinkage and selection via the lasso.
\newblock {\em J. R. Stat. Soc. Ser. B Stat. Methodol. 58}, 1 (1996), 267--288.

\bibitem{vanderPas2014}
{\sc van~der Pas, S., Kleijn, B., and van~der Vaart, A.}
\newblock The horseshoe estimator: Posterior concentration around nearly black
  vectors.
\newblock {\em Electron. J. Stat. 8\/} (2014), 2585--2618.

\bibitem{Yang2015}
{\sc {Yang}, Y., {Wainwright}, M.~J., and {Jordan}, M.~I.}
\newblock {On the computational complexity of high-dimensional Bayesian
  variable selection}.
\newblock arXiv:1505.07925, 2015.

\end{thebibliography}

\end{document}